\documentclass[reqno,12pt]{amsart}

\usepackage{amssymb,amsmath,psfrag}
\usepackage{enumerate}
\usepackage{graphicx}
\usepackage{hyperref}
\usepackage{mathrsfs}
\usepackage[usenames]{xcolor}

\newcommand\sectionpage{}
\renewcommand\sectionpage{\newpage}

\newtheorem{lem}{Lemma}[section]
\newtheorem{cor}[lem]{Corollary}
\newtheorem{prop}[lem]{Proposition}
\newtheorem{thm}[lem]{Theorem}

\theoremstyle{definition}

\newtheorem{exam}[lem]{Example}
\newtheorem{construction}[lem]{Construction}
\newtheorem{qn}[lem]{Question}

\numberwithin{equation}{section}
\numberwithin{table}{section}
\numberwithin{figure}{section}

\newcommand\cl{\operatorname{cl}}

\newcommand\bbF{\mathbb{F}}

\newcommand\bbQ{\mathbb{Q}}
\newcommand\bbR{\mathbb{R}}
\newcommand\bbZ{\mathbb{Z}}

\newcommand\HP{\mathrm{HP}}
\newcommand\Hc{\mathrm{H}}

\newcommand\cF{{\mathcal F}}
\newcommand\cI{{\mathcal I}}

\newcommand\cL{{\mathcal L}}

\newcommand\cP{{\mathcal P}}

\newcommand\fG{\mathfrak{G}}

\renewcommand\ell{l}
\newcommand\bF{\mathbf{F}}
\newcommand\bK{\mathbf{K}}
\newcommand\PR{\ensuremath{\mathbb{PR}}}
\newcommand\bV{\mathbf{V}}
\newcommand\PP{\Pi}
\newcommand\h{\mathbf{h}}
\newcommand\hi{\mathbf{h}^\infty}

\renewcommand\setminus\smallsetminus
\newcommand\inv{^{-1}}

\begin{document}
\pagestyle{myheadings}
\markboth{\sc Fl\'orez and Zaslavsky}{\sc Projective Rectangles: Harmonic Conjugation}


\thispagestyle{empty}

\title{Projective Rectangles: Harmonic Conjugation}

\author{Rigoberto Fl\'orez}
\thanks{Fl\'orez's research was partially supported by a grant from The Citadel Foundation.}
\address{Dept.\ of Mathematical Sciences, The Citadel, Charleston, South Carolina 29409}
\email{\tt rigo.florez@citadel.edu}

\author{Thomas Zaslavsky}
\address{Dept.\ of Mathematics and Statistics, Binghamton University, Binghamton, New York 13902-6000}
\email{\tt zaslav@math.binghamton.edu}

\date{\today}

\begin{abstract}
A projective rectangle is like a projective plane that has different lengths in two directions.  We develop harmonic conjugation in projective rectangles.  We construct projective rectangles in some harmonic matroids (matroids where harmonic conjugation is defined on every collinear point triple), such as Desarguesian projective planes of finite characteristic, by harmonic conjugation from extended lift matroids based on finite fields.  Similar results follow for countable fields with characteristic $0$.  We also show that projective rectangles are almost harmonic matroids.
\end{abstract}



\maketitle

\thispagestyle{empty}

\tableofcontents

\sectionpage\section {Introduction}

We generalize the concept of a projective plane to a projective rectangle, which is like a projective plane that has different sizes in two directions.  Projective rectangles include projective planes as trivial examples; more broadly they are substructures found in harmonic matroids (matroids in which harmonic conjugation is well defined), in particular in full algebraic matroids, which is where we first encountered them.

Nontrivial projective rectangles have many properties not generally found in projective planes. For example, they have remarkable internal structure, as we showed in \cite{pr1} and summarize in Section \ref{background}, and they are almost, but apparently not quite, harmonic matroids (see Theorem \ref{thm:harmonicconjugation}).  
In this paper we develop the line of thought that originally inspired projective rectangles, namely the first author's exploration of harmonic conjugation as applied to a certain matroid we call $L_p$ (a type of complete lift matroid to be defined later in this article, closely related to the well-known Reid cycle matroid).  We develop a theory of harmonic conjugation in harmonic matroids (which include Desarguesian and Moufang projective planes as well as full algebraic matroids) and show that multidimensional generalizations of $L_p$,  generate projective rectangles by means of iterated harmonic conjugation.  Our main conclusion is the existence of nontrivial projective rectangles in some harmonic matroids.  We also show that a multidimensional generalization of the Reid cycle matroid is a small matroid such that, when embedded in a harmonic matroid, it generates a projective rectangle.

We wish to acknowledge the deep influence of Bernt Lindstr\"om's ground-breaking work on abstract harmonic conjugation \cite{ldt, lhc}, which, as extended by the first author in \cite{rfhc}, inspired this whole project.

\sectionpage\section {Background about projective rectangles}\label{background}

An \emph{incidence structure} is a triple $(\cP,\mathcal{L},\mathcal{I})$ of sets with $\mathcal{I}
\subseteq \cP \times \mathcal{L}$.  The elements of $\cP$
are \emph{points}, the elements of $\mathcal{L}$ are \emph{lines}.
A point $p$ and a line $l$ are \emph{incident} if $(p,l) \in \mathcal{I}$.  A set $P$ of points is said to be
\emph{collinear} if all points in $P$ are in the same line.  We say that two
distinct lines \emph{intersect in a point} if they are incident with the same point.

A \emph{projective rectangle} is an incidence structure $(\cP,\mathcal{L},\mathcal{I})$ that satisfies the following axioms:

\begin{enumerate} [({A}1)]
\item \label{Axiom:A1}   Every two distinct points are incident with exactly one line.

\medskip

\item \label{Axiom:A2} There exist four points with no three of them collinear.

\medskip

\item \label{Axiom:A3}  Every line is incident with at least three distinct points.

\medskip

\item \label{Axiom:A4}  There is a \emph{special point} $D$.
A line incident with $D$ is called \emph{special}.  A line that is not incident with $D$ is called \emph{ordinary}, and a point that is not $D$ is called \emph{ordinary}.

\medskip

\item \label{Axiom:A5}  Each special line intersects every other line in exactly one point.

\medskip

\item \label{Axiom:A6}  If two ordinary lines $l_1$ and $l_2$ intersect
in a point, then every two lines that intersect both  $l_1$ and $l_2$ in four
distinct points, intersect in a point.

\end{enumerate}

\emph{Notation}: We write \PR\ for a projective rectangle.  We write $\overline{pq}$ for the unique line that contains two points $p$ and $q$.  After we establish the existence of projective planes in \PR, we use the notation $\overline{abc\dots}$ to mean the unique line (if $abc\dots$ are collinear) or plane (if they are coplanar but not collinear) that contains the points $abc\dots$.

\begin{exam}\label{ex:L2k}
The matroid $L_2^k$ is an example of a projective rectangle
(see Figure \ref{figure1}).  It has $m+1=3$ special lines.  Let $A:= \{ a_g \mid g \in (\bbZ_2)^k\}
\cup \{D \}$, $B:=\{ b_g \mid g \in (\bbZ_2)^k\} \cup \{D \}$ and
$C:=\{ c_g \mid g \in (\bbZ_2)^k\} \cup \{D \}$.  Let
$L_2^k$ be the simple matroid of rank 3 defined on the ground set
$E:= A\cup B\cup C$ by its rank-2 flats.  The non-trivial rank-2 flats are $A$, $B$, $C$, which are the special lines,
and the sets $\{a_g, b_{g +h}, c_h \}$ with $g$ and $h$ in $(\bbZ_2)^k$, which are the ordinary lines.

\begin{figure} [htbp]
\begin{center}
\includegraphics[width=8cm]{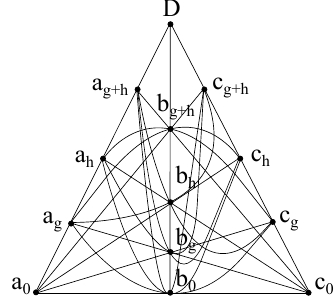}
\caption{The matroid $L_2^2$ with group $\bbZ_2 \times \bbZ_2$.} 
\label{figure1}
\end{center}
\end{figure}
\end{exam}

If a projective rectangle $\PR$ with exactly $m+1$ special lines has one of them with $n+1$ points, then we say that the \emph{order} of $\PR$ is $(m,n)$.  We do not assume $m$ or $n$ is finite unless we so state. 
The following result states basic properties from \cite{pr1}.

\begin{thm} \label{numberofpoinsinlines} 
If\/ $\PR$ is a projective rectangle of order $(m,n)$, then the following hold:
\begin{enumerate}[{\rm(a)}]

\item  \label{partitionPR:i} The ordinary points are partitioned by all special lines deleting $D$.

\item  \label{numberofpoinsinlines:g} $n \geq m$.

\item  \label{numberofpoinsinlines:c} All ordinary lines have $m+1$ points.

\item  \label{numberofpoinsinlines:f} All special lines have $n+1$ points.

\item  \label{numberofpoinsinlines:d} There are $m+1$ special lines and $n^2$ ordinary lines.  

\item  \label{numberofpoinsinlines:e} There are $n+1$ lines incident with each ordinary point.

\item  \label{numberofpoinsinlines:b} If $l$ is a line and $p$ is a point not in $l$, then the number of lines incident with $p$ intersecting $l$ equals the number of points on $l$.

\item \label{cor:numberofpoinsinlines:b} For a given point $p$ in an ordinary line $l$, there are $n-1$ ordinary lines intersecting $l$ at $p$.

\end{enumerate}
\end{thm}

A projective rectangle is a combination of projective planes, in the strong sense that every two intersecting ordinary lines are lines of a substructure that is a projective plane.  
For this, we clarify the notion of substructure of an incidence structure $(\cP,\cL,\cI)$.
An \emph{incidence substructure} of $(\cP,\cL,\cI)$ is an incidence structure $(\cP',\cL',\cI')$ in which $\cP' \subseteq \cP$, $\cL' \subseteq \cL$, and the incidence relation is the restriction of that in the superstructure.  If $(\cP',\cL',\cI')$ is a projective plane, it may contain an ordinary line and all its points; we call that kind \emph{full}.  A subplane need not be full; it also need not be a maximal subplane.  However:

\begin{prop}\label{prop:maximalsubplane}
In a projective rectangle, every maximal projective subplane is full.
\end{prop}

From now on, by a \emph{plane} in a projective rectangle we mean a full projective subplane.  Also, when we say some points or lines are \emph{coplanar}, we mean there is a plane $\pi$ such that each of the points is a point of $\pi$, each line that is ordinary is a line of $\pi$, and for each line $s$ that is special, $s \cap \pi$ is a line of $\pi$.

\begin{thm}\label{prop:twolinesintersectingpp} 
Let $\PR$ be a projective rectangle.  The special point $D$ is a point of every plane.  
If $s$ is a special line in the projective rectangle $\PR$ and $\pi$ is a plane in $\PR$, then $s \cap \pi$ is a line of $\pi$.

If two ordinary lines in $\PR$ intersect in a point, then both lines are lines of a unique plane in $\PR$.

Given three noncollinear ordinary points in a projective rectangle $\PR$, there is a unique plane in $\PR$ that contains all three points.
Given an ordinary line $l$ and an ordinary point $p$ not in $l$, there is a unique full projective plane in $\PR$ that contains both.
\end{thm}

A projective rectangle is nontrivial if and only if it contains more than one plane.  

\begin{cor}\label{coro;lineinaprojectiveplane}
Let $\PR$ be a projective rectangle.  Every ordinary line in $\PR$ is a line of a plane in $\PR$.  If\/ $\PR$ is nontrivial, then every ordinary line $l$ is a line of at least three planes that contain $l$.
\end{cor}

\begin{prop}\label{threelinesintersect}
Suppose three lines in a projective rectangle $\PR$ intersect pairwise in three different points.  Then they are a coplanar triple.

Equivalently, if three lines intersect pairwise (i.e., are pairwise coplanar) but are not a coplanar triple, then they all intersect in the same point.
\end{prop}

\sectionpage\section{Harmonic conjugation}\label{sec:harmonicconjugation}

The first key concept of this paper was that of a projective rectangle.  
We now define the second: abstract harmonic conjugation (following Lindstr\"om \cite{lhc} and Fl\'orez \cite{rfhc}).
The main result in this section, Theorem \ref{thm:harmonicconjugation}, says that a nontrivial projective rectangle has harmonic conjugation almost everywhere.  

Let $d,e,f$ be three points in a matroid $M$ that are collinear in a line $l$.  Suppose
there are four more points $a_1,a_2,a_3,a_4$ in $M$ such that $\{d,a_1,a_2\}$,
$\{d,a_4,a_3\}$, $\{e,a_4,a_2\}$, $\{f,a_4,a_1\}$, and $\{f,a_3,a_2\}$ are collinear, while
$\{a_1,e,a_3\}$ may or may not be collinear.  Then, the seven points
form a \emph{harmonic configuration} $\HP(d,e,f;a_1,a_2,a_3,a_4)$
based on $(d,e,f)$.  There may be a point $h$ that lies
at the intersection of the lines spanned by $\{a_1,a_3\}$ and $\{d,f\}$,
so that $\{a_1,a_3,h\}$ and $\{d,e,f,h\}$ are
collinear.  If the point $h$ always exists and is the same
point for any harmonic configuration based on $\{d,e,f\}$, then
$h$ is the \emph{harmonic conjugate} of $e$ with respect to
$d$ and $f$.  Note that $e$ and $h$ are not necessarily
distinct.  We define $\Hc(d,f;e,h)$ to mean that there are elements
$d, e, f, a_1, a_2, a_3, a_4$ with $\HP(d, e, f; a_1, a_2, a_3, a_4)$ and
that $h$ is the harmonic conjugate of $e$ with respect to $d$ and $f$.

\begin{figure}[htb]
\includegraphics[width=5cm]{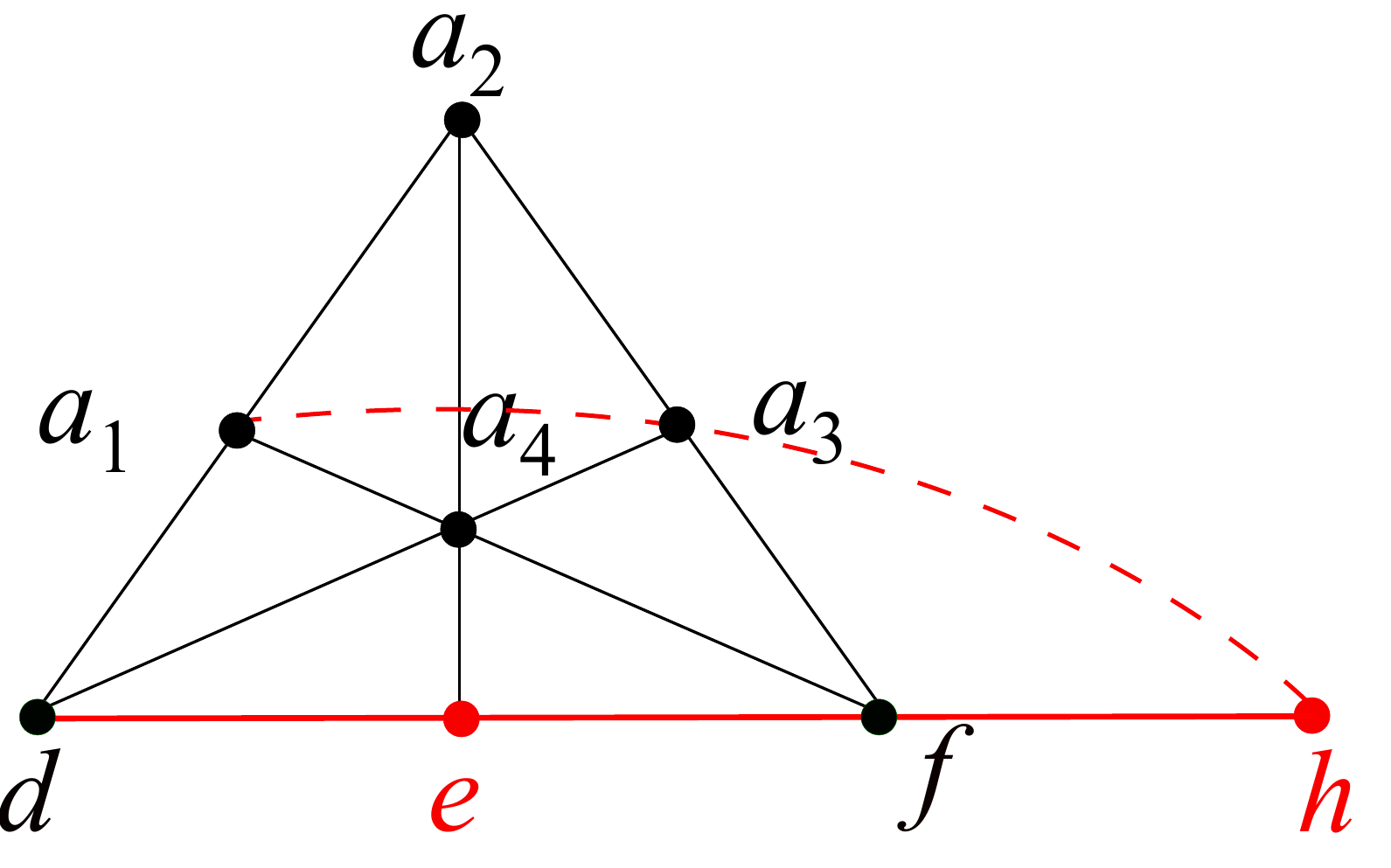}
\caption{The points and lines of a harmonic configuration, with the harmonic conjugate $h$.}
\label{F:harmonicgen}
\end{figure}

\begin{lem}\label{lem:harmonicconfigplane}
A harmonic configuration in a projective rectangle is contained in a unique plane.
\end{lem}

\begin{proof}
At least two lines $l_1,l_2$ of the harmonic configuration $S$ are ordinary; these two lines intersect so they are contained in a unique plane $\pi$, by Theorem \ref{prop:twolinesintersectingpp}.  
There are four kinds of line pair in $S$: outer pairs like $l, \overline{da_1a_2}$ which intersect at $d$, inner pairs like $\overline{da_4a_3}, \overline{fa_4a_1}$ which intersect at $a_4$, mixed pairs like $l, \overline{da_4a_3}$ which intersect at $d$, and mixed pairs like $l, \overline{ea_4a_2}$, which intersect at $e$.  (The first three types are the same under automorphisms of $S$.)  
The lines $l_1$ and $l_2$, being ordinary, are entirely contained in $\pi$, so their five points of $S$ are in $\pi$.  
At most three lines of $S$ can be special (and only if one of the points of $S$ is $D$) so there is a third ordinary line, which is also entirely in $\pi$ and whose additional point (or points) are therefore in $\pi$.  That leaves at most one point $p$ that could conceivably not lie in $\pi$, but $p$ is the intersection of two lines of $S$, say $l_3, l_4$.  If either of those lines is ordinary, say $l_3$, it is entirely in $\pi$ so $p \in l_3 \subseteq \pi$.  If neither of them is ordinary, then $p=D$, but $D$ is in every plane of the projective rectangle (Theorem \ref{prop:twolinesintersectingpp}), so $p \in \pi$.  In any case, all points of $S$ are in $\pi$.
\end{proof}

\begin{thm}\label{thm:harmonicconjugation}
In a nontrivial projective rectangle, the harmonic conjugate exists for any three points in an ordinary line and for any three collinear points of which one is the special point.
\end{thm}

We do not know whether the harmonic conjugate must exist if the three points lie on a special line but none is the special point.

\begin{proof}
Consider a nontrivial projective rectangle $\PR$ and a harmonic configuration $S$ of seven points as in Figure \ref{F:harmonicgen} (excluding $h$).  Since $S$ contains (at least) two ordinary lines that intersect in a point, it is in a plane $\pi$.

Suppose there is a second harmonic configuration $S'$ with corresponding points $a_i'$, $d^{\prime}=d$, $e^{\prime}=e$, and $f^{\prime}=f$, in a different plane $\pi'$ (see Figure \ref{F:harmonic2planes}).
From the figure we can see that the lines $l_i=\overline{a_ia_i'}$ for $i \in \{1,2,4\}$ are pairwise coplanar (for example, $l_1, l_2 \in \overline{a_1a_{1}^{\prime}d}= \overline{a_2a_2^{\prime}d}$).  From this and Proposition \ref{threelinesintersect}, we infer that $l_i=\overline{a_ia_i'}$ for $i \in \{1,2,4\}$ intersect in a point $r$.  Similarly, we prove that $l_i=\overline{a_ia_i^{\prime}}$ for $i \in \{2,3,4\}$ are pairwise coplanar.  So, by Proposition \ref{threelinesintersect}, $l_i=\overline{a_ia_i^{\prime}}$ for $i \in \{2,3,4\}$ intersect in a point $t$.  Since $l_2 \cap l_4=t$ and $l_2 \cap l_4=r$ are in both $l_2$ and $l_4$ they must be equal.  This implies that $l_1$ and $l_3$ are coplanar.  Therefore, the lines $\ell_{13}=\overline{a_1a_3}$ and $\ell_{13}^{\prime}=\overline{a_1^{\prime}a_3^{\prime}}$ are coplanar (see Figure \ref{F:harmonic4lines}).  This proves that $\ell_{13}$, $\ell_{13}^{\prime}$, and $l$ are pairwise coplanar.  From Proposition \ref{threelinesintersect} we obtain the desired result.

Now suppose we have two harmonic configurations in the same plane $\pi$, call them $S$ and $S''$, with corresponding points  $a_i''$, $d''=d$, $e''=e$, and $f''=f$.  First we assume $l$ is an ordinary line.  We construct a harmonic configuration $S'$ not in $\pi$ as follows.  Choose a point $a_2' \notin \pi$ (whence $a_2' \neq D$); the lines $da_2'$, $ea_2'$, and $fa_2'$ exist by Axiom (A\ref{Axiom:A1}) and there is a unique plane $\pi'$ containing two of them by Theorem \ref{prop:twolinesintersectingpp}.  Because $d,e,f$ are collinear, they are all in that same plane.  Then by applying the first, noncoplanar case twice we arrive at the conclusion.

If $l$ is a special line we have to be more careful.  We assume one of $d,e,f$ is the special point.  Suppose first that $e=D$.  Then $\overline{da_2'}$ and $\overline{fa_2'}$, being ordinary, determine a plane $\pi'$ that necessarily contains $D = e$.  The line $\overline{ea_2'}$ exists as a special line in $\pi'$.  Pick $a_3' \in \overline{ea_2'}$, different from $e$ and $a_2'$; the ordinary line $\overline{da_3'}$ intersects $\overline{ea_2'}$ at a point we name $a_4'$.  The line $\overline{da_4'a_e''}$ is in $\pi'$, so $a_4' \in \pi'$.  Thus, finally, the line $\overline{fa_4'}$ exists and intersects $\overline{da_2'}$ at a point we call $a_1'$.  This completes the construction of $S'$.

If $f=D$ (and similarly if $d=D$), the details are slightly different.  The lines $\overline{da_2'}$ and $\overline{ea_2'}$ are ordinary so they determine a plane $\pi'$.  The special line $\overline{fa_2'}$ meets $\pi'$ in a line of $\pi'$.  Pick $a_1' \in \overline{da_2'}$.  Then $\overline{fa_1'}$ meets $\overline{ea_2'}$ because the lines are lines of $\pi'$.  This point is $a_4'$.  Now $\overline{da_4'}$ intersects $\overline{fa_2'} \cap \pi'$ because both are lines of $\pi'$.  The point of intersection is $a_3'$.  Thus, the harmonic configuration $S'$ exists in this case.
\end{proof}

\begin{figure}[htb]
\includegraphics[width=7cm]{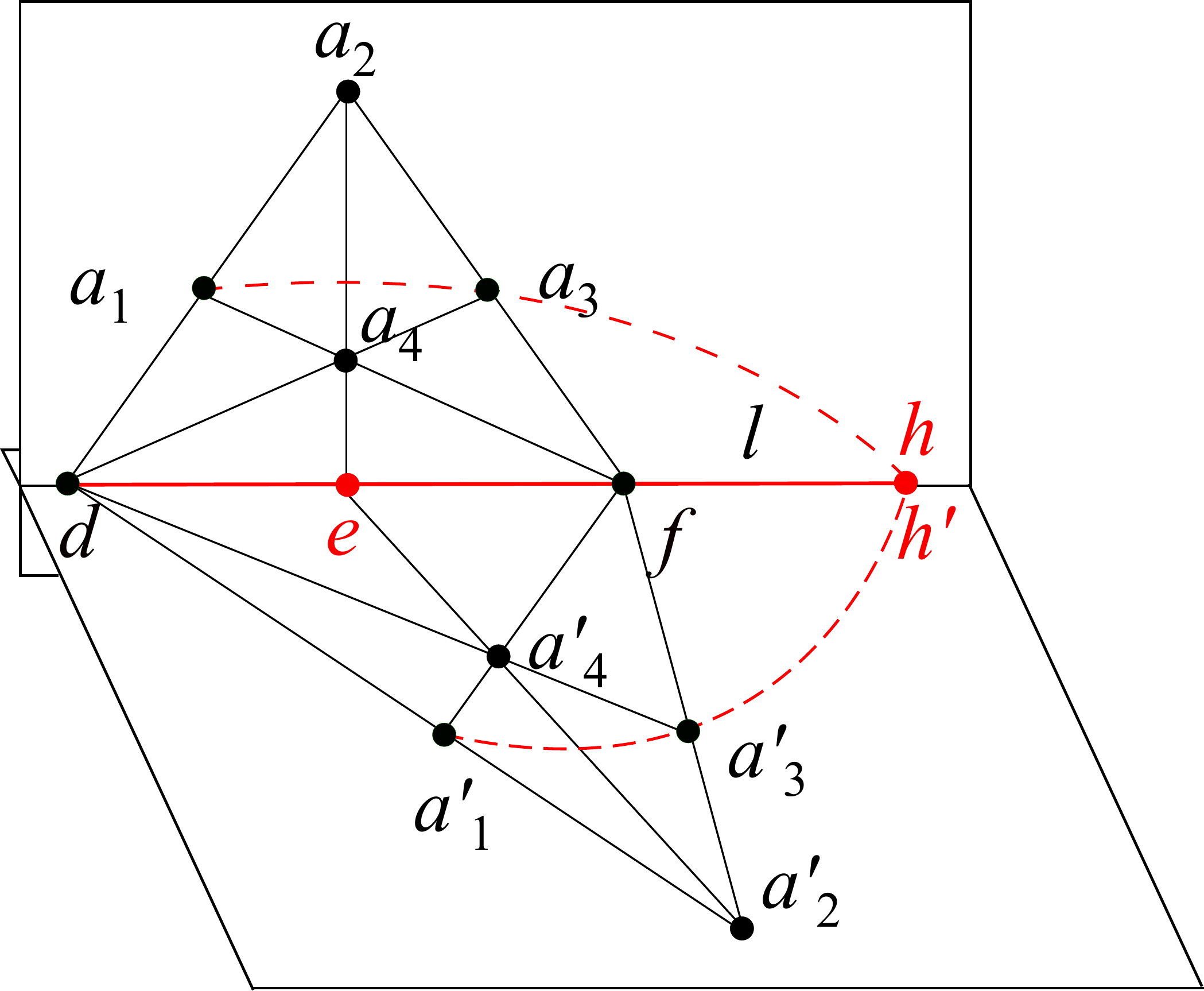}
\caption{Harmonic configurations in two planes that contain $l$, with a harmonic conjugate $h=h'$.}
\label{F:harmonic2planes}
\end{figure}

\begin{figure}[htb]
\includegraphics[width=7cm]{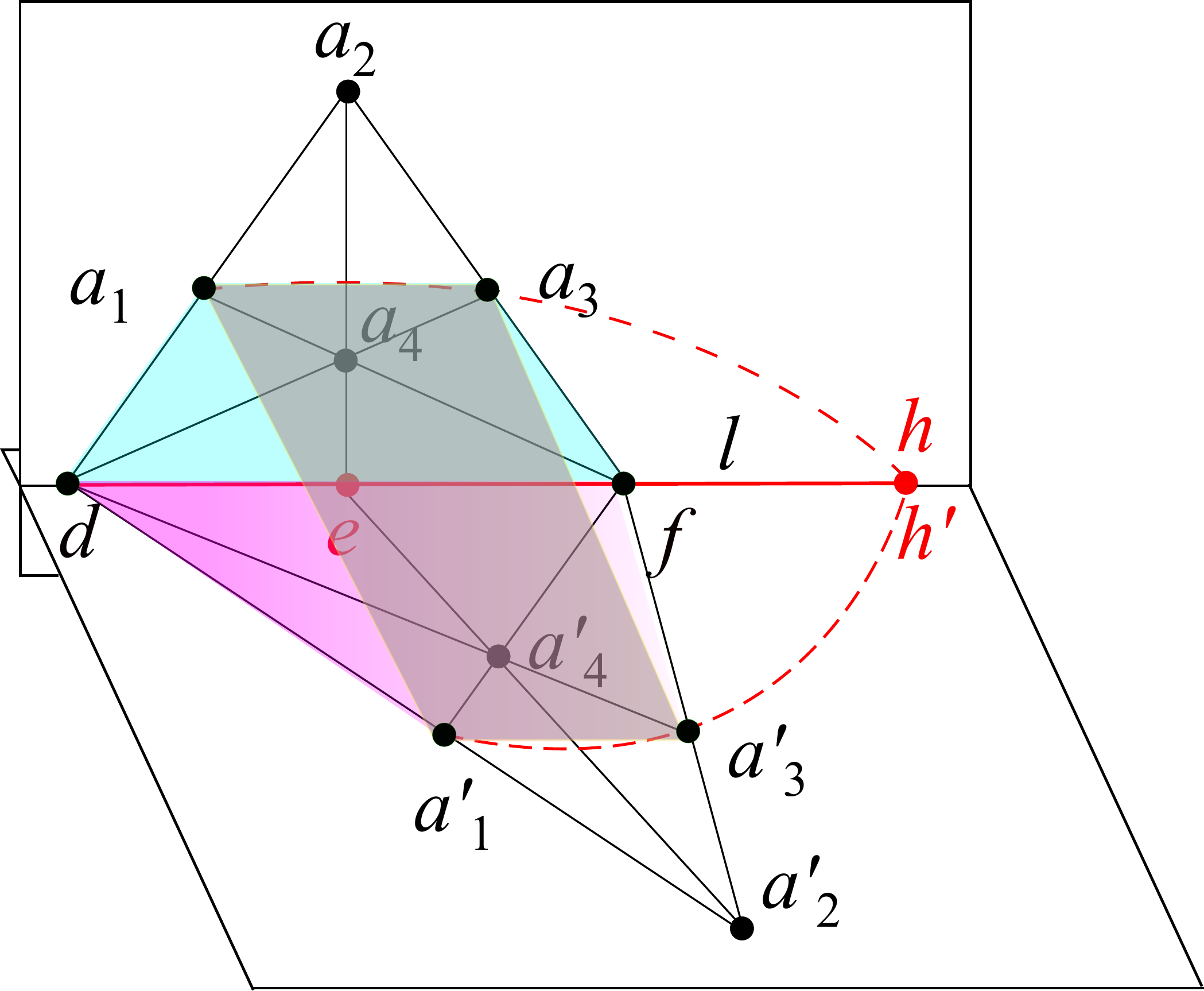}
\caption{The lines and planes between two harmonic configurations in two planes.}
\label{F:harmonic4lines}
\end{figure}

\sectionpage\section{Harmonic matroids}\label{sec:harmonicmatroids}

We now come to the third of our three key concepts: that of harmonic matroids (introduced in \cite{rflc}).
A matroid $M$ is \emph{harmonic} if a harmonic conjugate exists
for every harmonic configuration in $M$.
Familiar examples of harmonic matroids include Desarguesian planes
and all projective spaces of rank four or more.  Also, a
\emph{little Desargues plane} (also called a \emph{Moufang} or
\emph{alternative} plane) is a harmonic matroid;
in fact, a projective plane is a harmonic matroid if and only if it is a little Desargues plane \cite[page 202]{st}.  It follows that any projective plane contained in a harmonic matroid is a little Desargues plane.  

Full algebraic matroids are harmonic; Lindstr\"om \cite{lhc} defined them and proved this.  We repeat the definition here.
Let $\bF$ and $\bK$ be two algebraically closed fields such that $\bK$ is an extension of $\bF$ that has finite
transcendence degree over $\bF$.  The \emph{full algebraic
matroid} $A(\bK/\bF)$ has ground set $\bK$; a set $S \subseteq \bK$ is 
independent if and only if its elements are algebraically independent over $\bF$.  
A full algebraic matroid $A(\bK/\bF)$ of rank 3 has the following flats:
the field $\bF$ (the flat of rank 0, i.e., transcendence degree zero over $\bF$),
atoms or points (algebraically closed exensions of transcendence degree one over $\bF$), lines
(algebraically closed exensions of transcendence degree two), and one plane, the field $\bK$.

We shall discuss \emph{algebraic embedding} of a matroid $M$ over a
field $\bF$; that means a mapping $f$ from the elements of $M$
into an extension field of $\bF$, such that a subset $S$ is
independent in $M$ if and only if $f|_S$ is injective and $f(S)$ is
algebraically independent over $\bF$.  (Algebraic closure is not required for this concept.)

Theorem \ref{thm:harmonicconjugation} says a nontrivial projective rectangle is almost an example of a harmonic matroid, but the theorem leaves open the question of harmonic conjugates for triples that do not include $D$ but lie on a special line.  Does the theorem extend to all coplanar collinear triples?  
If so, all planes in the projective rectangle are little Desargues planes, even if the projective rectangle is not algebraically embeddable.  This is an open question.

The purpose of this section is to develop the machinery of modular sequences, which are a kind of finite or infinite sequence generated by repeated harmonic conjugation.

We write $\cl$ for the closure operator of a matroid.

\begin{lem} [{\cite[Proposition 2(ii) (proof) and Corollary 3]{rfhc}}]\label{tecnico}
If $\HP(y, x, z; o, q, r, s)$ and $\Hc(y,z ; x,x')$ in a harmonic matroid, then

\begin{enumerate}[{\rm(i)}]

\item\label{tecnicoPart1}  $\Hc(z,r ; q,q_1)$ and $\Hc(y,o ; q,q_2)$ for some points $q_1$
and $q_2$ in $\cl(\{ s, x'\})$,

\item\label{tecnicoPart2} $\Hc(o,r ; u, x')$ for some point $u$ in $\cl ( \{q, x \} )$.

\item\label{tecnicoPart3} $\Hc(y,r;s,t)$ for some point $t$ that is collinear
with $ q $ and $ x'$.

\end{enumerate}
\end{lem}

Let $\mathcal {P} (M)$ be the power set of the ground set of a
harmonic matroid $M$.  We define $ \h \colon \mathcal {P} (M)
\rightarrow  \mathcal {P}(M) $ by $ \h(S) = S \cup \{ x \in M \colon
\Hc(a,c ; b, x) \text { for some } a, b, c \in S \}$, the iterates $\h^k (S) = h
(\h^{k-1} (S))$, $ \hi (S) = \bigcup _{n\ge 0} \h^k(S).$ We
call  $\hi (S)$ the \emph{harmonic closure} of $S$ in $M$.
Note that $\hi $ is an abstract closure operator.  If $ \hi (S) = S$, then we say $S$ is \emph{harmonically closed}.  For
$S\subseteq M$, $S\subseteq \hi (S) \subseteq \cl (S)$ so
rank$(S)=$ rank$(\hi(S))$.

For the remainder of this paper we work with a vector space $\bV$ over the field $\bbZ_p$ 
where $p$ is a prime number $p$ ($p$ always denotes a prime)---or in Section \ref{Infinitecase} over the field $\bbQ$.

We now define the matroid $L_p^\bV$.  Let
$$
A:=\{ a_g \mid g \in \bV \} \cup \{D \},\ B:=\{ b_g \mid g \in \bV \}
\cup \{D \}, \text{ and } C:= \{ c_g \mid g \in \bV \} \cup\{D \},
$$ 
where all points $a_g, b_g, c_g, D$ are distinct.  
Let $L_p^\bV$ be the simple matroid of rank 3 on the ground
set $E:= A\cup B\cup C$ defined by the $p^{2\dim\bV}+3$ rank-2
flats $A$, $B$, $C$ and $\{a_g, b_{g+ h}, c_h \}$ with $g$ and $h$ in $\bV$.  
When $k=\dim\bV$ is finite, we may write $L_p^k$ for $L_p^\bV$.  
An example of this family of matroids is the matroid $L_2^2$ described above.  Because in characteristic 2 every point is its own harmonic conjugate, $L_2^k$ is harmonically closed: $\hi(L_2^k)=L_2^k$.  We denote $L_p^1$ by $L_p$.  
(The fact that $L_p^k$ is the complete lift matroid $L_0((\bbZ_p)^k{\cdot}K_3)$ of the group expansion of the triangle graph $K_3$ by the group $(\bbZ_p)^k$---see \cite[Section 3]{b2}---was part of the inspiration for the present paper.)

Fl\'orez \cite{rfhc} proved that the matroid 
$R_{\mathrm{cycle}}[p]:=L_p\setminus \{c_2, \dots , c_{p-1} \}$ extends by
harmonic conjugation to $L_p$.  The matroid $R_{\mathrm{cycle}}[p]$ is known
as the \emph{Reid cycle matroid} \cite[page 52]{jk}.  In Section
\ref{sec:minimalmatroid} we generalize the definition of $R_{\mathrm{cycle}}[p]$
to $R_{\mathrm{cycle}}^k [p]$ and prove this matroid extends to $\hi(L_p^k)$.

We now introduce conjugate sequences in a harmonic matroid \cite{rfhc}.  
A \emph{conjugate sequence} is a sequence of points $A,a_0, \dots ,
a_m, \dots $ such that $\Hc(A,a_{i};a_{i-1},a_{i+1})$ for $i \in \bbZ_{> 0}$.  
We observe that the points of a conjugate sequence are collinear.  
A conjugate sequence is
\emph{modular} if at least one of its points is equal to one of the preceding points $a_i$.  
We say that $m$ is the \emph{order} of a modular sequence $A,a_0, a_1, \dots, a_m, \dots$   
if $a_m$ is the first point of the sequence that is equal to a preceding point.  The smallest possible value of $m$ is $2$, since three distinct collinear points are needed in order to have a harmonic confguration.
(Note that we use the word ``order" for both the ``order" of a modular sequence and the ``order" of a projective plane.)
For a modular sequence of order $m$ we use the short notation $A,a_0, a_1, \dots , a_{m-1}$, but the sequence is implicitly infinite, so we still have all harmonic conjugacies $\Hc(A,a_{i};a_{i-1},a_{i+1})$ for $i=m, m+1, \dots$.  
In Proposition \ref{modularsequence} Part \eqref{modularsequencem0} we prove that if $A,a_0, a_1, \dots , a_{m-1}$ is a modular sequence of order $m$, then $a_m=a_0$, $a_{m+1}=a_1$, \dots, i.e., a modular sequence is periodic.

For instance, consider $m=2$.  Then $a_2=a_0$ and the modular sequence is $A,a_0,a_1,a_0,a_1,\dots$.  The harmonic conjugacy property $\Hc(A,a_{i};a_{i-1},a_{i+1})$ becomes $\Hc(A,a_1;a_0,a_0)$ if $i$ is odd and $\Hc(A,a_0;a_1,a_1)$ if $i$ is even.  An example is harmonic conjugation in the Fano plane.

Another example is illustrated by Figure \ref{F:tecnico1} (in different notation).  The modular sequence is $a_0,b_0,c_0,x_2,x_3,x_4,\dots$ (with $a_0$ in the role of $A$).  In this example $m=5$ and $x_5=b_0$, $x_6=c_0$, $x_7=x_2$, etc.

\begin{lem}[Fl\'orez \cite{rfhc}]\label{modmobius}
A modular sequence in a harmonic matroid $M$ satisfies $\hi (S)= S$.
\end{lem}

\begin{figure}[htbp]
\begin{center}
\includegraphics[scale=.6]{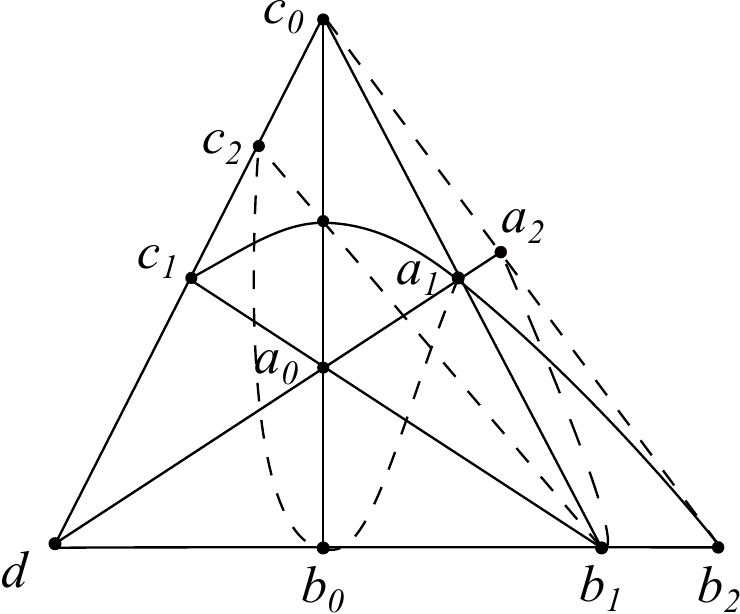}
\caption{The harmonic configuration $T$ in Proposition \ref{modularsequence} for $m=3$ (solid lines) along with the additional points $a_i$, $b_i$, and $c_i$ (where $i=2=m-1$).}
\label{F:T}
\end{center}
\end{figure}

\begin{prop}\label{modularsequence} Suppose that $\hi(d, b_0, b_1) = (d, b_0, \dots , b_{m-1})$
is a modular sequence of order $m$ in a harmonic matroid $M$.  
If $c_1, c_0, a_1,a_0$ are points of $M$ with $T:=\HP(d, b_0, b_1; c_{1},c_0, a_1,a_0)$, 
then the following properties hold:
\begin{enumerate}[{\rm(i)}]
\item \label{modularsequencemprime} $m$ is a prime number.

\item \label{modularsequencem0} We have $b_{m+i} = b_i$ for all $i\geq0$.

\item\label{modularsequence:i} $\pi:=\hi(T)$ is a Desarguesian projective plane
of order $m$ in $M$.

\item\label{modularsequence:iii} There are points $a_0, \dots , a_{m-1}$ in $M$ such that
$d,a_0, \dots , a_{m-1}$ is a modular sequence of order $m$ and $\{a_i, b_i, c_0\}$ is a
collinear set.  

\item\label{modularsequence:5} $\hi(d, b_0, b_1)$ is a line of $\pi$.

\item\label{modularsequence:ii}  If $l$ is a line in $T$, then $\hi(l)$ is a line of $\pi=\hi(T)$.

\end{enumerate}
\end{prop}

\begin{proof} 
Much of the proposition can easily be proven by modifying the hypothesis in the proof of Proposition 16 in \cite{rfhc}, but since it plays an  important role in this paper we include a complete proof.

Let $d, b_0, \dots , b_{m-1}$ be all the points of the modular sequence  
$\hi(d, b_0, b_1)$, so $\Hc(d,b_1;b_0,b_2)$, $\Hc(d,b_2;b_1,b_3)$, $\dots, \Hc(d,b_{m-1}; b_{m-2},b_m)$   
and $b_m$ is the first point that equals one of its preceding points.

Most of the proof is based upon a property $P(t)$, proved using mathematical induction.  For $t \in \bbZ_{\geq0}$, let  $P(t)$  
be the statement that there is a point $a_t$ in $M$ such that $\{a_t,b_t, c_0\}$, $\{a_t,b_{t+1},c_1\}$ and  
$\{d,a_0, \dots, a_t\}$ are collinear sets.  $P(0)$ is the statement that $T$ exists.  $P(1)$ is the statement that $b_2$ is the harmonic conjugate of $b_0$ with respect to $d$ and $b_1$, obtained by the configuration $T$, which follows from the assumption that $d, b_0, b_1, \dots$ is a conjugate sequence.

We now suppose that $t \geq 2$ and that $P(t-2)$ and $P(t-1)$ hold, and we deduce $P(t)$.
From  the modular sequence $\hi(d, b_0, b_1)$, the configuration $\HP(d, b_0, b_1; c_1, c_0, a_1, a_0)$, $P(t-2)$, and $P(t-1)$ 
we deduce that
$$\HP( d, b_{t-2}, b_{t-1}; c_1, c_0, a_{t-1}, a_{t-2}).$$
From $P(t-1)$ we deduce that $\{c_1,a_{t-1},b_t\}$ is a collinear set.  Thus, $b_t$ is the intersection of the line $\overline{c_1a_{t-1}b_t}$ with the line $\overline{\hi(d, b_0, b_1)}$, which implies $\Hc(d,b_{t-1};b_{t-2},b_t)$.   
Therefore, by 
Lemma \ref{tecnico} Part \eqref{tecnicoPart3} there is a point $a_t$ such that $ \{a_t, b_{t}, c_0\}$ and 
$\{d, a_{t-2}, a_{t-1}, a_t\}$ are collinear sets.  Thus, this configuration holds: 
$$\HP( d, b_{t-1}, b_t; c_1, c_0, a_t, a_{t-1}).$$
It follows that there is a point $b_{t+1}$ such that $\Hc(d,b_t;b_{t-1},b_{t+1})$ and $\{c_1,a_t,b_{t+1}\}$ is a
collinear set.  These prove $P(t)$.  

Most of Part \eqref{modularsequence:iii} follows because we have the configuration $\HP(d,a_{t-2},a_{t-1};c_1,c_0,b_{t-1},b_t)$ and collinearity of $\{c_1,b_{t-1},a_t\}$ and of $\{d, a_{t-2}, a_{t-1}, a_t\}$, which imply the harmonic conjugacy $\Hc(d,a_{t-1};a_{t-2},a_t)$.  We note that $b_t=b_i$ if and only if $a_t=a_i$ because of the collinear sets $\{c_0,a_t,b_t\}$ and $\{c_0,a_i,b_i\}$ (from $P(t)$ and $P(i)$, respectively) and the fact that the lines $\overline{da_0a_1\dots}$ and $\overline{db_0b_1\dots}$ are distinct. 

We now prove Parts \eqref{modularsequencemprime} and \eqref{modularsequencem0} and that $\hi (\HP(d, b_0, b_1; c_1, c_0, a_1, a_0))$ is a projective plane of prime order.

By the definition of harmonic conjugation, $b_m \ne d$.  Suppose that $b_m = b_t$ for some 
$t \in \{1, \dots , m-1\}$.  Since $\{c_1,a_{t-1},b_{t}\}$ and  $\{c_1,a_{m-1},b_m \}=\{c_1,a_{m-1},b_t \}$ are collinear (by $P(t-1)$ and $P(m-1)$, respectively), the set $\{c_1,a_{t-1}, a_{m-1} \}$ is collinear.  
Since $a_{t-1} \ne a_{m-1}$ (because $b_{t-1} \ne b_{m-1}$, because $t-1<m-1<m$), we conclude that $c_1$ is in the line $\overline{a_{t-1} a_{m-1}} = \overline{a_0a_1}$.  
That is a contradiction because $\{c_1,a_0, a_1 \}$ is not collinear in $\HP(d, b_0, b_1; c_1, c_0, a_1, a_0)$.  It follows that $b_m \neq b_t$ for $0<t<m$, so $b_m = b_0$.  We conclude that $a_m=a_0 \ne a_t$ for $0<t<m$; thus, $d, a_0,a_1,\dots$ has order $m$.

That proves the first step of Part \eqref{modularsequencem0}.  For the rest, for $t\geq0$ let $Q(t)$ be the statement that $b_{t+m} = b_{t+j}$ where $0 \le j < m$ if and only if $j=0$.  We apply induction on $t\ge0$.  $Q(0)$ was just proved.  Let $t>0$ and consider the harmonic relations $\Hc(d,b_{t-1};b_{t-2},b_t)$ and $\Hc(d,b_{m+t-1};b_{m+t-2},b_{m+t})$, which is the same as $\Hc(d,b_{t-1};b_{t-2},b_{m+t})$ by the induction assumption.  Thus, $b_t=b_{m+t}$ are equal by uniqueness of the harmonic conjugate.  Now, compare $b_{t+m}=b_t$ to $b_{t+j} = b_{t-m+j}$; they are unequal by induction.  That completes the proof of Part \eqref{modularsequencem0}.  Since $d, b_0, b_1, \dots$ is any modular sequence of order $m$, the same conclusion applies to the sequence $d,a_0,a_1,\dots$.

Note that $M|E$ where $E= \{d,a_0, \dots , a_{m-1}, b_0, \dots ,
b_{m-1}, c_0, c_1 \}$ is the matroid $R_{\text{cycle}}[m]$.  Thus
$R_{\text{cycle}}[m]$ is embedded in a harmonic matroid.  By
\cite[Theorem 10]{rflc}, $m$ must be a prime number  (this proves Part \eqref{modularsequencemprime}). 
By the definition of $E$ we deduce that
$$\pi:= \hi ( \{d, b_0, b_1; c_1, c_0, a_1, a_0\}) = \hi (M|E) =  \hi (R_{\text{cycle}}[m]).$$ Therefore, by
\cite[Proposition 16]{rfhc}, $\hi(R_{\text{cycle}}[m])$ is a projective plane of order $m$.
The plane is a little Desargues plane because it is contained in the harmonic matroid $M$, and since it is finite, it is a Desarguesian plane coordinatized by a finite field \cite[Theorem 6.20]{HP}, specifically the prime field $\bbZ_m$.

Part \eqref{modularsequence:5} follows from the fact that $\hi(d, b_0, b_1)$ is a collinear set in $\pi$ of $m+1$ points, the size of a line in $\pi$ by Part \eqref{modularsequence:i}.

We prove Part \eqref{modularsequence:ii}.  
From Parts \eqref{modularsequencemprime} and \eqref{modularsequence:i} we know that $l$ is a set of collinear points in the Desarguesian plane $\pi$ of prime order $m$.  Let $\bar(l)$ denote the line of $\pi$ that contains $l$.  If $l$ does not have three points, add one point of $\bar(l)$ so we have three collinear points.  Let those points be called $d', b_0', b_1'$.  Extend $d', b_0', b_1'$ to a configuration $T'$ isomorphic (as a matroid) to $T$ with the correspondence $d', b_0', b_1' \leftrightarrow d, b_0, b_1$.  Extend these three points to a conjugate sequence in $\pi$, which is necessarily finite, hence modular of order $m'$.  By Part \eqref{modularsequence:i}, $\hi(T')$ is a projective plane $\pi'$ of order $m'$ and is contained in $\pi$.  By Part \eqref{modularsequence:5}, $\hi(d', b_0', b_1')$ is a line of $\pi'$.  By projective geometry, $\pi'=\pi$; thus, $\hi(d', b_0', b_1')$ is a line of $\pi$.
\end{proof}

The proof of Part \eqref{modularsequence:ii} establishes a stronger statement, which may be a known theorem, but we have not seen it.

\begin{cor}\label{conjugatesequence-p}
Any conjugate sequence in a Desarguesian plane of prime order $p$ is modular of order $p$.
\end{cor}

\sectionpage\section{A projective rectangle in a harmonic matroid}\label{sec:PRinharmonic}

In this section we show that if $L_p^\bV$ (where $p$ is prime) is embedded in a harmonic matroid, its harmonic closure is a projective rectangle.  Thus, if a harmonic matroid $M$ contains $L_p^\bV$, then it contains a projective rectangle.  
(From \cite[Section 4]{b4} it is known that $L_p^\bV$ does embed in some harmonic matroids, if $\bV$ is the additive group of a field of characteristic $p$; namely, the projective plane over that field.)  
In particular, if  $L_p$ embeds in $M$, then $M$ contains the Desarguesian projective plane of order $p$.  
We state the main theorem, whose proof appears at the conclusion of the section.

\begin{thm} \label{ThmHarmonicPR}
For each embedding of $L_p^\bV$ in a harmonic matroid, $\hi(L_p^\bV)$ is a projective rectangle of order $(p,|\bV|)$ in which every plane is Desarguesian of order $p$.
\end{thm}

We assume throughout this section that $L_p^\bV$ is embedded in a harmonic matroid $M$.

We often write $\hi(a_g,c_h)$ with only two points instead of $\hi(a_g,b_{g+h},c_h)$, because the point $b_{g+h}$ is predictable.
For convenience in formulas we employ the function  $\phi_{\epsilon}: \bbZ_p \to \bV$, which is an injection defined by  $\phi_{\epsilon}(t)= t\epsilon-tg+h$ where $g, h,$ and $\epsilon$ are fixed elements of  
$\bV$ such that $\epsilon - g \neq 0$.  For the sake of simplicity, and if there is not ambiguity, we write $\phi(t)$ instead of $\phi_{\epsilon}(t)$.

\begin{lem} \label{tecnico1} 
Let $g$ and $h$ be elements of $\bV$.  For every $t \in \{1,\dots,p-1\}$ there is a point $x_{t+1}$ in $M$ such that

\begin{enumerate} [{\rm(i)}]

\item \label{tecnico1Part1} for every $\epsilon  \in \bV$,
$\{a_{\epsilon}, b_{\epsilon+ \phi(t)}, c_{\phi(t)}, x_{t+1} \}$ is collinear;

\item \label{tecnico1Part2} $\hi\{a_g, b_{g+h}, c_h\} = \{a_g, b_{g+h}, c_h, x_2, \dots , x_{p-1}\}$ is a modular sequence of order $p$.  

\end{enumerate}

\end{lem}

Recall that $x_p=b_{g+h}$ by the definition of a modular sequence.

\begin{figure}[htbp]
\begin{center}
\includegraphics[scale=.7]{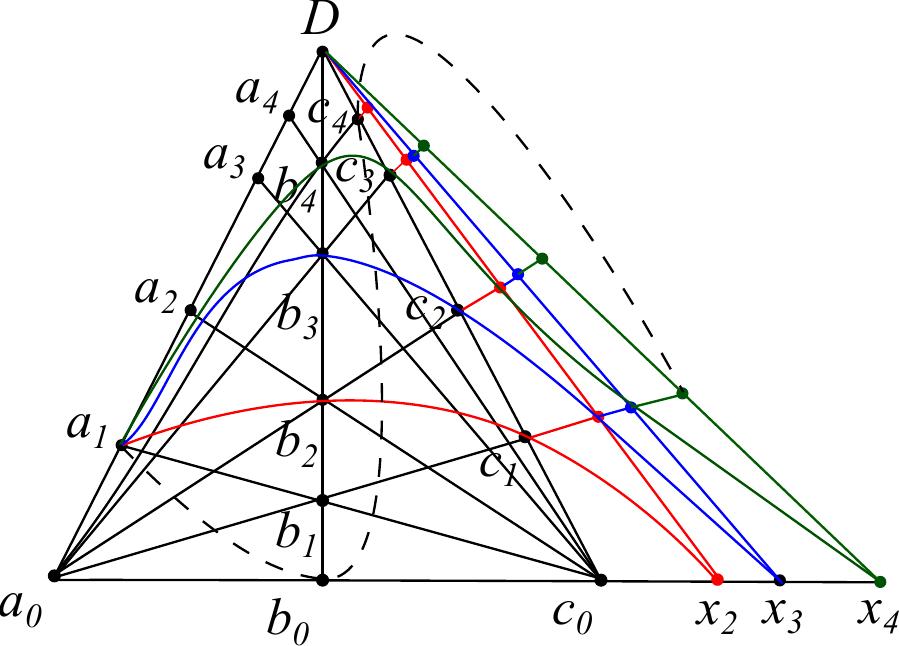}
\caption{The construction of $x_t$'s in Lemma \ref{tecnico1}.  In this example $k=1$, $p=5$, $g=h=0$, and the configuration of the points $D,a_i,b_i,c_i$ is the matroid $L_5$.  The $a_{i\epsilon}, b_{i\epsilon}, c_{i\epsilon}$ in the lemma's proof are simplified to $a_i, b_i, c_i$.}
\label{F:tecnico1}
\end{center}
\end{figure}

We simplify the proof by applying an automorphism.

\begin{lem} For any $\alpha, \beta \in \bV$, the mapping $a_g, b_{g+h}, c_h \mapsto a_{g+\alpha}, b_{g+h+\alpha+\beta}, c_{h+\beta}$ is an automorphism of $L_p^\bV$.\end{lem}  
\begin{proof} A line $a_g b_{g+h} c_h$ becomes a line $a_{g+\alpha} b_{(g+\alpha)+(h+\beta)} c_{h+\beta}$. \end{proof}

Take $\alpha = -g$ and $\beta = -h$; this converts the general case of Lemma \ref{tecnico1} to the case $g=h=0$ and gives $\phi(t) = t+\epsilon$.  (We think of the automorphism as relabelling $L_p^\bV$, not moving points, so there is no need for it to extend to an automorphism of $M$.)  We restate Lemma \ref{tecnico1}:

\begin{enumerate} [{\rm(i)}]
\item \label{tecnico1Part1rev} For every $\epsilon  \in \bV$, $\{a_{\epsilon}, b_{(t+1)\epsilon}, c_{t\epsilon}, x_{t+1} \}$ is collinear;
\item \label{tecnico1Part2rev} $\hi\{a_0, b_{0}, c_0\} = \{a_0, b_{0}, c_0, x_2, \dots , x_{p-1}\}$ is a modular sequence of order $p$.  
\end{enumerate}

Now we are in a copy of $L_p$ with $A_\epsilon=\{a_{i\epsilon}: i\in \bbZ_p\}$, $B_\epsilon=\{b_{i\epsilon}: i\in \bbZ_p\}$, $C_\epsilon=\{c_{i\epsilon}: i\in \bbZ_p\}$.  That reduces the proof to $L_p$.

\begin{proof}[Proof of Lemma \ref{tecnico1}]
It is convenient to define $x_0=b_0$ and $x_1=c_0$.  

First of all, we prove both parts of the lemma for $p=2$.  By the definition of $L_2^k$, $\{a_{0}, b_{0} , c_{0}\}$ and $\{a_{\epsilon}, b_{2\epsilon} , c_{\epsilon} \}$ 
are collinear sets.  Since $2\epsilon=0$ in $\bV$ over $\bbZ_2$, taking $x_2 = b_{0} = x_0$ and $x_3=c_{0}=x_1$ completes the proof for $p=2$.

We now prove the lemma for $p > 2$.  We use mathematical induction on $t \in \bbZ_p$, where here we treat $\bbZ_p$ as a field and let $t = 1,2,\dots,p$ so that $(t-1)\inv$ is meaningful for $t>1$.  To simplify some formulas we define $x_0=b_0$ and $x_1=c_0$.  

Suppose that $t=1$.  From the definition of $L_p^\bV$ we deduce that the configuration 
\begin{equation}\label{formula0}
 \HP(a_0, b_0, c_0; a_{\epsilon}, D, c_{\epsilon}, b_{\epsilon})
\end{equation}
holds in $M$.  This implies that there is a point $x_2 \in M$ with 
$\Hc(a_0, c_0; b_0, x_2)$ such that $\{a_{\epsilon}, b_{2\epsilon}, c_{\epsilon}, x_{2}\}$
and $\{a_{0}, b_0, c_{0}, x_{2}\}$ are collinear sets.  We observe that $x_2$ depends only on $a_0,b_0,c_0$ and is independent of $\epsilon$, because $M$ is a harmonic matroid.  
This proves the proposition for $t=1$.

We now fix $t \in \{2, \dots, p \}$ and suppose that there are points $x_{t-1}$ and $x_{t}$ in $M$ such that
\begin{equation}\label{formula1}
\{a_{\epsilon},b_{(t-1)\epsilon}, c_{(t-2)\epsilon}, x_{{t-1}}\},
\end{equation}
\begin{equation}\label{formula2}
\{a_{\epsilon},b_{t\epsilon} , c_{(t-1)\epsilon}, x_{t}\},  \text{ and }
\end{equation}
\begin{equation}\label{formula3}
\{a_{0}, b_0, c_{0}, x_{2}, \dots, x_{t}\}
\end{equation}
are collinear sets.  
From the definition of the matroid $L_p^\bV$ we know that the sets
$\{a_{0}, c_{t\epsilon}, b_{t\epsilon}\}$ and $\{a_{0}, a_{\epsilon\cdot((t-1)^{-1}+1)}, a_{\epsilon}\}$ are collinear.

Taking $\epsilon\cdot{ (t-1)^{-1}(t)}$ instead of $\epsilon$ in \eqref{formula1} and \eqref{formula2}, we deduce that
$\{a_{\epsilon\cdot{ (t-1)^{-1}(t)} }, b_{t\epsilon}, x_{t-1}\}$ and 
$\{a_{\epsilon\cdot{ (t-1)^{-1}(t)} }, c_{t\epsilon }, x_{t}\}$ are collinear sets.  These two sets are respectively equal to 
$\{a_{\epsilon\cdot((t-1)^{-1}+1)} , b_{t\epsilon}, x_{t-1}\}$ and $\{a_{\epsilon\cdot((t-1)^{-1}+1)} , c_{t\epsilon}, x_{t}\}$.  
These,  \eqref{formula2}, and  \eqref{formula3} imply that
$$
\HP(a_0, x_{t-1},x_{t}; c_{t\epsilon}, b_{t\epsilon}, a_{\epsilon}, a_{\epsilon\cdot((t-1)^{-1}+1)})
$$
holds in $M$.  Therefore, there is a point $x_{t+1} \in M$ with 
$\Hc(a_0, x_{t}; x_{t-1}, x_{t+1})$ such that $\{a_{\epsilon}, c_{t\epsilon}, x_{t+1}\}$, hence $\{a_{\epsilon}, b_{(t+1)\epsilon}, c_{t\epsilon}, x_{t+1}\}$, and  $\{a_0, b_0, c_0, x_2, \dots , x_{t+1}\}$  
are collinear sets.

We have proved $\Hc(a_0,x_{t};x_{t-1},x_{t+1})$ for $t=1, \dots, p$ and that 
$$ \{a_0, b_0, c_0, x_2, \dots , x_{p}\} \subseteq \hi(\{a_0, b_0, c_0\}),$$
hence $\{ a_0, b_0, c_0, x_2, \dots , x_p, x_{p+1} \}$ is a collinear set and $a_0,b_0,c_0,x_2,\dots,x_{p+1}$ is a conjugate sequence.

Since $p\epsilon= 0$, we have $b_{p\epsilon}  = b_0$.  
We deduced at $t=p-1$ that the set $\{a_{\epsilon},b_{p\epsilon} , c_{(p-1)\epsilon}, x_{p}\}$ 
is collinear and lies in a distinct line from $\hi(\{a_0, b_0, c_0\})$.  
Therefore,
$$b_0,x_p \in\{a_{\epsilon},b_{p\epsilon} , c_{(p-1)\epsilon}, x_{p}\} \cap \hi(\{a_0, b_0, c_0\}).$$ 
Since both $b_0$ and $x_p$ are in the intersection of distinct lines, $x_p=b_0$.  
This proves that $a_0, b_0, c_0, x_2, \dots , x_{p-1}$ is a modular sequence of order at most $p$.  By Proposition \ref{modularsequence} Part \eqref{modularsequencem0}, the order divides $p$, which is prime; thus the sequence has order $p$.  
Since by Lemma \ref{modmobius} a modular sequence is harmonically closed,
$\hi\{a_0, b_0, c_0\} \subseteq \{a_0, b_0, c_0, x_2, \dots , x_{p-1}\}$.  

This completes the proof.
\end{proof}

The following construction of a projective rectangle in $M$ is similar to that used by Fl\'orez in \cite{rfhc}, where from an embedding of $L_p$ into a harmonic matroid he produced an embedding of the Desarguesian plane of order $p$.  (His analog of the vector space $\bV$ was simply $\bbZ_p$.)

\begin{construction}[Construction of the Harmonic Closure]\label{C:Rt}
We construct the harmonic closure $\hi(L_p^\bV)$, step by step, by applications of Lemma \ref{tecnico}.  We begin with $L_p^\bV = A \cup B \cup C$ and we define $R_\infty = A$, $R_0=B$, and $R_1=C$.  Given $R_0, \ldots, R_{t}$ for a particular $t=1,2,\dots,p-2$, we define $R_{t+1}$ by harmonic conjugation from $R_\infty, R_{t-1}, R_t$.  The conjugate sequence $\hi(a_0,b_0,c_0) = (a_0,x_0=b_0,x_1=c_0,\ldots,x_{t-1},x_t,x_{t+1},\dots)$ gives us $x_{t+1}$ as the first element of $R_{t+1}$.  

When $t=1$, $R_\infty \cup R_{t-1} \cup R_{t}$ is $L_p^\bV$.  We assume in the induction that each $R_i \setminus D$ for $i\leq t$ is labeled by $\bV$, the labels being $x_{t,h}$, in such a way that $R_\infty \cup R_{t-1} \cup R_{t}$ is isomorphic to $L_p^\bV$ by a mapping that is the identity on $R_\infty$ and maps $b_g, c_h$ respectively to $x_{t-1,g}, x_{t,h}$.  This isomorphism means that the construction of $R_t$ is identical to that of $R_2$, only beginning with a different isomorph of $L_p^\bV$.  Thus, we only need to construct $R_2$.

There were several ways to construct $x_2$, the harmonic conjugate of $b_0$.  We formed a harmonic configuration $\HP(a_0,b_0,c_0;a_g,D,c_h,b_k)$ for some $g, h, k \neq 0$ and observed that the line $\overline{a_gc_h} = \hi(a_g,b_{g+h},c_h)$ intersects $\hi(a_0,b_0,c_0)$ at the harmonic conjugate of $b_0$ with respect to $a_0$ and $c_0$, because we are in the harmonic matroid $M$.  The point is $x_2$.  In fact, $h=k=g$ because of the lines $a_gb_kc_0$ and $a_0b_kc_h$ in $L_p^\bV$, which imply $k=g+0=0+h$.  
Because we are in a harmonic matroid, we get the same harmonic conjugate $x_2$ for every choice of $g$.

Next, we construct the other points in $R_2$.  By Lemma \ref{tecnico} Part \eqref{tecnicoPart3}, the conjugate sequence $\hi(a_0,b_h,c_h)$ intersects $\overline{Dx_2}$ in a point that we call $x_{2,h}$, which is the harmonic conjugate of $b_h$ with respect to $a_0$ and $c_h$.  (This is where we set up the labeling of $R_2$.)  

Now pick any $g' \in \bV$ and let $b_{h+g'}$ be the intersection of the line $\overline{a_{g'}c_h}$ in $L_p^\bV$ with $B$.  The harmonic configuration $\HP(a_0,b_h,c_h; a_{g'},D,c_{h'},b_{h'+h})$ implies that the harmonic conjugate $x_{2,h}$ of $b_h$ with respect to $a_0$ and $c_h$ is on the line $\overline{a_{g'}c_{h}}$; therefore, $\{a_{g'},b_{g'+h'},c_{h'},x_{2,h}\}$ is a collinear set; more precisely $(a_{g'},b_{g'+h'},c_{h'},x_{2,h},\ldots)$ is the modular sequence $\hi(a_{g'},c_{h'})$ and $x_{2,h}$ is the harmonic conjugate of $b_{g'+h'}$ with respect to $a_{g'},c_{h'}$.

We have produced an element of $R_2 \setminus D$ for every $h \in \bV$.  To establish the isomorphism of $A \cup C \cup R_2$ with $L_p^\bV$ we must show that, for arbitrary $g',h',h$, $a_{g'},c_{h'},x_{2,h}$ are collinear if and only if $h' = g'+h$.  Assume collinearity; we need to show that $h'=g'+h$.  Consider the point in $B$ at which $\overline{a_0c_{h'}}$ intersects $\overline{a_{g'}c_h}$.  In the former line it is $b_{0+h'}$ and in the latter it is $b_{g'+h}$; but these are the same point, so $h'=g'+h$.  In the opposite direction, if $h'=g'+h$ then $b_{0+h'} = b_{g'+h}$ so $\overline{a_0c_{h'}}$ intersects $\overline{a_{g'}c_h}$ at $b_{h'}$ and $x_{2,h}$ is the harmonic conjugate of $b_{g'+h'}$ with respect to $a_{g'},c_{h'}$; i.e., it is collinear with $a_{g'},c_{h'}$.  This completes the proof of isomorphism and permits us to advance to the next step of the induction.

Since the modular sequences all have length $p$, it is easy to see that $R_p=R_0$.
\end{construction}

To confirm the validity of the construction we use a labeling lemma.  Define $x_{0,g+h} = b_{g+h}$ and $x_{1,h} = c_h$ and for $t\geq0$ let $x_{t,g,h}$ be the intersection of $\hi(a_g,c_h)$ with $R_t$; thus, $x_{0,g,h} = x_{0,g+h}$ and $x_{1,g,h} = x_{1,h}$.

\begin{lem}\label{labelformula}
For each $t\in \{0,1,\dots, p\}$, we have $x_{t,g,h} = x_{t,(1-t)g+h}$.  
\end{lem}

\begin{proof}
We showed that the point $x_{2,g,h} = x_{2,-g+h}$ where $h = -g'+h'$ and $x_{2,g',h'}$ is the next point after $x_{1,h'}$ in the modular sequence $a_{g},\ldots,x_{1,h'},\dots$.  This means that at each future stage, when the sequence is $a_{g},\ldots,x_{t-1,h'},\dots$, we have $x_{t,g,h} = x_{t,-g+h'}$.  By induction, $h'=(1-[t-1])g+h$, and that establishes the formula.
\end{proof}

Lemma \ref{labelformula} and observations in the construction enable us to specify collinearity of points $x_{i,f}$.

\begin{cor}[Common Modular Sequence]\label{tecnicocollinearcor}
\begin{enumerate}[{\rm(i)}]
\item\label{tecnicoD} For $t, t' \in \{0,1,\dots, p\}$, the points $x_{t,g,h}$ and $x_{t',g',h'}$ are collinear with $D$ if and only if $t|t'$.

\item\label{tecnicoequal}
For each $t\in \{0,1,\dots, p\}$, we have $x_{t,g_1,h_1} = x_{t,g_2,h_2}$ if and only if $(1-t)g_1+h_1 = (1-t)g_2+h_2$.

\item\label{tecnicocollinear} Points $a_g$ and $x_{t,f}$, where $t \geq 0$, are in the unique common modular sequence $\hi(a_g,c_h) = (a_g,b_{g+h},c_h,\dots)$ where $h=(t-1)g+f$.

\item\label{tecnicocollinear:ii} Points $x_{t,f}$ and $x_{t',f'}$, where $t' > t \geq 0$, are in the unique common modular sequence $\hi(a_g,c_h)$ where $g = [f-f']\cdot{(t'-t)\inv}$ and $h=(t-1)g+f = (t'-1)g+f'$.

\item\label{tecnicocollinear:iii} Points $a_g, x_{t,f}, x_{t',f'}$, where $t' > t > 0$, are in a common modular sequence $\hi(a_g,c_h)$ if and only if $f-f' = (t-t')g$.  Then $h=(t-1)g+f = (t'-1)g+f'$.

\item\label{tecnicocollinear:iv} Points $x_{t,f}, x_{t',f'}, x_{t'',f''}$, where $t'' > t' > t \geq 0$, are in a common modular sequence $\hi(a_g,c_h)$ if and only if $[f-f']\cdot{(t'-t)\inv} = [f'-f'']\cdot{(t''-t')\inv}$.  Then $g = [f-f']\cdot{(t'-t)\inv}$ and $h=(t-1)g+f = (t'-1)g+f' = (t''-t)g+f''$.
\end{enumerate}
\end{cor}

\begin{prop}\label{P:Rt}
The set $\cP$ is the harmonic closure of $L_p^\bV$.  The lines in $\cP$ are $R_\infty=A, R_0=B, R_1=C, R_2, \ldots, R_{p-1}$ (with $R_p=R_0$) and the modular sequences $\hi(a_g,b_{g+h},c_h)$ for $g,h \in \bV$.
\end{prop}

\begin{proof}
This summarizes the result of the inductive construction.  To complete the proof we need to show the point labels are the same in $R_p$ as in $R_0$.  That follows from setting $t=p$ in Lemma \ref{labelformula}.
\end{proof}

We note the following consequence of the construction.

\begin{cor}\label{tecnicomod}
For every $g \in \bV$, $t \in \{0,1,\dots, p-1\}$ and $x_t \in R_t \setminus D$, there exists a modular sequence of the form $\hi(a_g,c_h)$ that contains $x_t$.
\end{cor}

A final corollary to formalize the basic tool of the construction.

\begin{cor}\label{tecnicolabeliso2} 
Define $\lambda: R_t \to \bV$ by $\lambda(x_{t,g,h}) = (1-t)g+h$ for all $t \in \{0,1,\ldots,p-1\}$.
The submatroids of $M$ formed by $R_\infty \cup R_{t-1} \cup R_{t}$ and, for each $0 \le t'' < t' < t \le p-1$ by $R_{t''} \cup R_{t'} \cup R_{t}$ are isomorphic to $L_p^\bV$ by the labeling $\lambda$ with $R_\infty \mapsto A$, $R_{t-1} \mapsto B$, and $R_{t} \mapsto C$. 
\end{cor}

\begin{lem}\label{tecnico2} Let $g, h, q$ and $r$ be elements of
$\bV$.  The set $l_1 = \hi(a_g,c_h)$ is a line
of a Desarguesian projective plane of order $p$ in $\hi(L_p^\bV)$.  Moreover,
if $l_1$ intersects the line $l_2 = \hi(a_q,c_r)$ in a single point,
then both are lines of a Desarguesian projective plane of order $p$ in $\hi(L_p^\bV)$.
\end{lem}

\begin{proof}
First we prove  the lemma for $p=2$.  Since  $\hi(L_2^\bV)=L_2^\bV$ in this case, we have 
$l_1= \{ a_g, b_{g+h}, c_h\}$  and $l_2=\{a_q, b_{q+r}, c_r\}$. 
So, the intersection of $l_1$ and $l_2$ is either $\{a_g\}=\{a_q \}$, $\{b_{g+h}\}=\{b_{q+r} \}$  or
$\{c_h\}=\{c_r \}$. 

Suppose that $a_g = a_q $.  Since $g+r+2h=g+r$ and $g+2r+h=g+h$
in $\bV$, both sets $\{a_{g+r+h}, b_{g+h}, c_r\}$ and $\{a_{g+r+h}, b_{g+r}, c_h \}$ are collinear.  
Since these are lines that intersect only in $a_g$, we conclude that $r \ne h$.
The collinearities imply that the configuration 
$\HP(a_g, b_{g+h}, c_h; a_{g+r+h}, D, c_r, b_{g+r})$ holds in $L_2^\bV$ and is a Fano plane.
That proves there is a projective plane in $\hi(L_p^\bV)$ that
contains both $l_1$ and $l_2$.  The proof is similar if $c_h=c_r$.

Suppose that $b_{g+h} = b_{q+r}$.  This implies that $g+h =
q+r$.  Since $-h=h$ and $-r=r$, it follows that $b_{g+r}= b_{q+h}$ in $\bV$.
Therefore, $\HP(a_g, b_{g+h}, c_h; a_q, D, c_r, b_{g+r})$ is a
Fano plane in $L_2^\bV$ that contains both $l_1$ and $l_2$.

We now prove the lemma for $p>2$.  If the intersection of $l_1$ and $l_2$ is $\{a_g\}=\{a_q\}$, 
then, clearly, $\HP(a_g, b_{g+h}, c_h; a_{g+r-h}, D, c_r, b_{g+r})$ holds in
$L_p^\bV$.  This, Lemma \ref{tecnico1} Part \eqref{tecnico1Part2}, and 
Proposition \ref{modularsequence} imply that
$\hi(\HP(a_g, b_{g+h}, c_h; a_{g+r-h}, D, c_r, b_{g+r}))$ is a projective plane that contains both $l_1$ and
$l_2$.  If the intersection of $l_1$ and $l_2$ is $\{c_h\}=\{c_r \}$,
the proof is similar.  Therefore, we may now assume $q \ne g$ and $r \ne h$.

Since $l_1 =  \hi \{a_g, b_{g+h}, c_h\}$, Lemma
\ref{tecnico1} guarantees the existence of points $x_2, \dots ,
x_{p}$ in $M$, where $x_{p}=b_{g+h}$, such that the line $l_1=\{a_g, c_h, x_2, \dots , x_{p-1}\}$ is a modular sequence of order $p$.  
This and Proposition \ref{modularsequence} imply the first part of the lemma.

Suppose that the intersection  of $l_1$ and $l_2$ is $\{x_{j+1}\}$
for some $ j \in \{1, \dots, p-1 \} $.  This implies that
$\{a_q, c_r, x_{j+1}\}$ is collinear.  Lemma
\ref{tecnico1} with $\epsilon = q$ implies that $\{a_{q},
c_{\phi_{q}(j)}, x_{j+1} \}$ is collinear.  Thus, $\{a_{q},
c_{\phi_{q}(j)}, c_r, x_{j+1} \}$ is collinear.  The intersection of the two collinear sets
$\{a_{q}, c_{\phi_{q}(j)}, c_r, x_{j+1} \}$ and $C$ 
cannot have more than one point, so we know that $c_{\phi_{q}(j)}= c_r$.  Thus, $l_2$ is
equal to $\hi \{a_{q}, c_{\phi_{q}(j)}, x_{j+1} \}$.

Suppose $j \ne p-1$ and let $\Delta={g+(q-g)\cdot{j(j+1)^{-1}}}$.  
From the definition of  $L_p^\bV$ we know that 
$\{a_{g}, c_{\phi_{q}(j)}, b_{g\phi_{q}(j)} \}$ and 
$\{a_{g}, a_{q}, a_{\Delta} \}$ are collinear sets.  
Lemma \ref{tecnico1} Part \eqref{tecnico1Part1} implies both that 
$\{ a_{\Delta}, b_{g\phi_{q}(j)} , x_{j+1} \}$ is collinear, by taking $\epsilon = \Delta$ and computing $\Delta \phi_\Delta(j) = g+\phi_q(j)$, 
and that $\{ a_{q}, b_{q+\phi_{q}(j-1)}, x_j \}$
and $\{ a_{q},c_{\phi_{q}(j)} , x_{j+1} \}$ are collinear sets, 
by taking $\epsilon = q$.  These, the equality $q+\phi_{q}(j-1) = g+\phi_{q}(j)$, 
and the fact that $\{ a_g, x_j, x_{j+1} \}$ is collinear, imply that
$$\HP(a_g, x_{j},x_{j+1}; c_{\phi_{q}(j)}, b_{g\phi_{q}(j)}, a_{\Delta}, a_{q})$$ 
holds in $M$.  This and Proposition \ref{modularsequence} imply that
$$\hi ( \HP(a_g, x_{j},x_{j+1}; c_{\phi_{q}(j)}, b_{g+\phi_{q}(j)}, a_{\Delta}, a_{q}) )$$ 
is a projective plane that contains both $l_1$ and $l_2$.

If $j=p-1$ we have $\HP(D, a_1, a_g; b_{g+r}, b_{g+h}, c_h, c_r)$.  In this case $$\hi ( \HP(D, a_1, a_g; b_{g+r}, b_{g+h}, c_h, c_r) )$$ 
is a projective plane that contains both $l_1$ and $l_2$.
\end{proof}

\begin{prop}\label{mainprop} Let $l_1$ and $l_2$ be two distinct lines in $\hi(L_p^\bV)$
with $l_1 \cap l_2 \not= \emptyset$ and neither of them incident with $D$.  If two other distinct lines in  
$\hi(L_p^\bV)$ intersect both $l_1$ and $l_2$ in four distinct points, then those two lines intersect.
\end{prop}

\begin{proof} Suppose that $l_3$ and $l_4$ are two distinct lines in
$\hi(L_p^\bV)$, such that $l_3$ intersects $l_1$ and $l_2$ in
$\{y_1\}$ and $\{y_2\}$, respectively, and  $l_4$ intersects
$l_1$ and $l_2$ in $\{y_3\}$ and  $\{y_4\}$ respectively.

From Lemma \ref{tecnico2} we know that $l_1$ and $l_2$ lie in a
projective plane $\pi \subseteq \hi(L_p^\bV)$.  Since $y_1, y_2,
y_3$ and $ y_4$ are in $\pi$, the lines $l_3\cap \pi$ and $l_4\cap
\pi$ intersect in $\pi$.  This proves that $l_3 \cap l_4 \not = \emptyset$.
\end{proof}

\begin{construction}[Projective Rectangle by Harmonic Conjugation] \label{HarmonicConjPR}
Let $\PR:=(\cP,\mathcal{L},\mathcal{I})$, where 	
\begin{align*}
\cP \   &:={ \bigcup_{g,h \in \bV} \hi(a_g,c_h)} \cup \{D\} = R_{\infty} \cup {\bigcup_{i=0}^{p-1} R_i},\\
\cL  \ &:= \{R_i : i= \infty, 0, \dots, p-1\} \cup \{\hi(a_g,c_h) :  g,h \in \bV\},
\end{align*}
and $\mathcal{I}$ is the incidence relation of set membership. 
\end{construction}

\begin{thm} \label{tecnico3} The system $\PR$  is a projective rectangle of order $(p,|\bV|)$.
\end{thm}

\begin{proof} 
Clearly, $D$ serves as the special point and $R_i$ is a special line for each $i \in \{1,2, \dots, p-1,\infty\}$. Any point apart from $D$ belongs to some line $R_i$, hence it is collinear with $D$. This establishes that $\PR$ satisfies Axioms (A\ref{Axiom:A2}) and (A\ref{Axiom:A4}).

The construction of the sets of the form  $\hi(a_g,c_h)$ for $g,h \in \bV$ and the definition of the sets 
$R_{\infty}, R_0, \dots, R_{p-1}$ in Construction \ref{HarmonicConjPR} imply that every set in $\cL$ has at least $p+1 > 3$ elements, thus we seamlessly achieve the satisfaction of Axiom (A\ref{Axiom:A3}).

From Lemma \ref{tecnico1} we know that for fixed $g^{\prime},h^{\prime} \in \bV$ the set  $\hi(a_{g^{\prime}},c_{h^{\prime}})$ intersects $R_i$ in exactly one point for $i \in \{0, 1, \dots, p-1, \infty \}$.  
This and the fact that $R_i\cap R_{j}=\{D\}$ for $i \ne j$ prove that $\PR$ satisfies Axiom (A\ref{Axiom:A5}).  

Let $w$ and $q$ be two points in $\PR$, neither equal to $D$.  If both points are in $R_i$ for some $i$, they are collinear.  Now suppose that $w \in R_s$ and $q\in R_t$, for $s\ne t$.
Again, from the definition of $L_p^\bV$ and Lemma \ref{tecnico1} Part \eqref{tecnico1Part1}, there are
$|\bV|$ sets of the  form $\hi(a_g,c_h)$ for $g,h \in \bV$ that contain $w$.  
By Corollary \ref{tecnicocollinearcor} Part \ref{tecnicocollinear}  the set $\cP \setminus R_s$ is contained in the union of all $|\bV|$ sets of the  form $\hi(a_g,b_{g+h},c_h)$ that contain $w$.  
Since $q \in R_t \ne R_s$,  $q$ and $w$ must be collinear.  
This proves Axiom (A\ref{Axiom:A1}).

Let $l_1$ and $l_2$ be two distinct sets in $\cL$, where $l_1 \cap l_2 \neq \emptyset$ and neither of them is incident with $D$. Suppose that $l_3$ and $l_4$ are two distinct sets in $\cL$ such that $l_3$ intersects $l_1$ and $l_2$ at two distinct points, and $l_4$ intersects $l_1$ and $l_2$ at two other distinct points. According to Proposition \ref{mainprop}, they have non-empty intersection. 

From Construction \ref{C:Rt} we know that $R_i \cap \hi(a_g,c_h) \ne \emptyset$ for $g,h \in \bV$ and $i \in \{0,1,2, \dots, p-1, \infty\}$. 
This proves that $\PR$ satisfies Axiom (A\ref{Axiom:A6}) and thus completes the proof. 
\end{proof}

\begin{proof}[Proof of Theorem \ref{ThmHarmonicPR}]
The proof is a straightforward application of Construction \ref{HarmonicConjPR}, Theorem \ref{tecnico3}, and Proposition \ref{P:Rt}.
\end{proof}

\begin{cor} \label{CoroHarmonicPP}
For each embedding of $L_p^\bV$ in a harmonic matroid and each embedding of $L_p$ into $L_p^\bV$, $\hi(L_p)$ is a Desarguesian projective plane of order $p$ embedded in $\hi(L_p^\bV)$.
\end{cor}

\begin{proof} 
This is a consequence of the proof of Theorem \ref{tecnico3}.
\end{proof}

Let $\PP(q)$ denote the projective plane over the field $\bbF_q$.  We know from Theorem \ref{ThmHarmonicPR} with $\bV = (\bbZ_p)^k$ that $\PP(p^k)$ contains a projective rectangle $\hi(L_p^k)$, so any harmonic matroid $M$ that contains the plane $\PP(q)$ also contains a $\hi(L_p^k)$ as a submatroid of the plane.  On the other hand, Theorem \ref{ThmHarmonicPR} states that $M$ contains a rectangle $\hi(L_p^k)$, without reference to projective planes.  We do not know whether this $\hi(L_p^k)$ will be contained in a plane $\PP(p^k)$ in $M$.  It is conceivable that for every harmonic matroid $M$, every $\PR \subseteq M$ is contained in a projective plane in $M$.  This would make the appearance of projective rectangles a sidelight to that of projective planes.

\begin{qn}\label{rectangleinplane}
Is there a full algebraic matroid $M$, or any other harmonic matroid, that contains a projective rectangle that is not itself contained in a projective plane in $M$?

In particular, is the projective rectangle $\hi(L_p^k) \subseteq M$ of Theorem \ref{ThmHarmonicPR} with $\bV = (\bbZ_p)^k$ necessarily contained in a plane $\PP(p^k)$ in $M$?
\end{qn}

\begin{qn}\label{infinitefield}
Is Theorem \ref{ThmHarmonicPR} valid if $L_p^\bV$ is replaced by $L_\bF^\bV$ where $\bF$ is a field other than $\bbZ_p$, in particular an infinite field, and $L_\bF^\bV$ is the complete lift matroid $L_0(\bV{\cdot}K_3)$?  (See Section \ref{Infinitecase} for the rational field.)
\end{qn}

\sectionpage\section{The rational field $\bbQ$} \label{Infinitecase}

We apply the methods of the previous section to the smallest field with characteristic 0, namely, $\bF=\bbQ$ with additive group $\bbQ^+$ and vector space $\bV$ any rational vector space (for example, $\bbR$).  The same proof applies to any countable field whose characteristic is $0$.

As in the previous section we write $\hi(a_g,c_h) := \hi(a_g,b_{g+h},c_h)$; however, here the conjugate sequence is infinite instead of modular.  To express these sequences in terms of $\bbQ$ we need a fixed bijection $\psi:  \bbQ \rightarrow \bbZ_{\geq0}$.  
We use the function $\phi_{\epsilon}: \bbZ_{\geq0} \to \bV$, which is an injection defined by  $\phi_{\epsilon}(t)= t\epsilon-tg+h$ where $g, h,$ and $\epsilon$ are fixed elements of  
$\bV$ such that $\epsilon - g \neq 0$.  For the sake of simplicity, and if there is not ambiguity, we write $\phi(t)$ instead of $\phi_{\epsilon}(t)$.

\begin{lem} \label{tecnico61} 
Let $g$ and $h$ be elements of $\bV$.  If $L_0(\bV{\cdot}K_3)$ is embedded in a harmonic matroid $M$,
then for every $r \in \bV \setminus\{0\} $ there is a point $x_{r}$ in $M$ such that

\begin{enumerate} [{\rm(i)}]

\item \label{tecnico61Part1} for every $\epsilon  \in \bV$,
$\{a_{\epsilon}, b_{\epsilon+ \phi(r)}, c_{\phi(t)}, x_{r} \}$ is collinear;

\item \label{tecnico61Part2} $\hi\{a_0, b_{0}, c_0\} = \{a_0, b_{0}, c_0\} \cup \{ x_i \mid i\in \bV\}$ is a countable modular sequence.  

\end{enumerate}

\end{lem}

Before the proof we establish another lemma.

\begin{lem} For any $\alpha, \beta \in \bV$, the mapping $a_g, b_{g+h}, c_h \mapsto a_{g+\alpha}, b_{g+h+\alpha+\beta}, c_{h+\beta}$ is an automorphism of $L_0(\bV{\cdot}K_3)$.
\end{lem}  
\begin{proof} A line $a_g b_{g+h} c_h$ becomes a line $a_{g+\alpha} b_{(g+\alpha)+(h+\beta)} c_{h+\beta}$. \end{proof}

Take $\alpha = -g$ and $\beta = -h$; this converts the general case to the case $g=h=0$ and gives $\phi(t) = t\epsilon$.  (As in the finite case, we think of the automorphism as relabelling $L_0(\bV{\cdot}K_3)$.)  We restate Lemma \ref{tecnico61}:

\begin{enumerate} [{\rm(i)}]
\item \label{tecnico1Part1revQ} for every $\epsilon  \in \bV$, $\{a_{\epsilon}, b_{(t+1)\epsilon}, c_{t\epsilon}, x_{t+1} \}$ is collinear;
\item \label{tecnico1Part2revQ} $\hi\{a_0, b_{0}, c_0\} = \{a_0, b_{0}, c_0\} \cup \{ x_i \mid i\in \bbZ_{\geq0}\}$ is a countable conjugate sequence.
\end{enumerate}

Now we are in a copy of $L_0(\bV{\cdot}K_3)$ with $A_\epsilon=\{a_{i\epsilon}: i\in \bbQ\}$, $B_\epsilon=\{b_{i\epsilon}: i\in \bbQ\}$, and $C_\epsilon=\{c_{i\epsilon}: i\in\bbQ\}$.
  That reduces the proof to $L_0(\bbQ^+{\cdot}K_3)$.

\begin{proof}[Proof of Lemma \ref{tecnico61}] 

We use mathematical induction  on $t=\psi(r) \in \bbZ_{\geq0}$.  We define $x_0=b_0$ and $x_1=c_0$.  Suppose that $t=1$.  From the definition of $L_0(\bV{\cdot}K_3)$  
we deduce that the configuration 
\begin{equation}\label{formula60}
 \HP (a_0, b_0, c_0; a_{\epsilon}, D, c_{\epsilon}, b_{\epsilon} )
\end{equation}
holds in $M$.  This implies that there is a point $x_2 \in M$ with 
$\Hc(a_0, c_0; b_0, x_2)$ such that $\{a_{\epsilon}, b_{2\epsilon}, c_{\epsilon}, x_{2}\}$
and $\{a_{0}, b_0, c_{0}, x_{2}\}$ are collinear.  We observe that $x_2$ depends only on $a_0,b_0,c_0$ and is independent of $\epsilon$, because $M$ is a harmonic matroid.  
This proves Part \eqref{tecnico61Part1} and Part \eqref{tecnico61Part2} for $t=1$.

We now fix $t=\psi(r)$ for some $r \in  \bV$ and suppose that there are points $x_{t-1}$ and $x_{t}$ in $M$ such that
\begin{equation}\label{formula61}
\{a_{\epsilon},b_{(t-1)\epsilon}, c_{(t-2)\epsilon}, x_{{t-1}}\},
\end{equation}
\begin{equation}\label{formula62}
\{a_{\epsilon},b_{t\epsilon} , c_{(t-1)\epsilon}, x_{t}\},  \text{ and }
\end{equation}
\begin{equation}\label{formula63}
\{a_{0}, b_0, c_{0}, x_{2}, \dots, x_{t}\}
\end{equation}
are collinear sets.  
From the definition of the matroid $L_0(\bV{\cdot}K_3)$ we know that the sets
$\{a_{0}, c_{t\epsilon}, b_{t\epsilon}\}$ and $\{a_{0}, a_{\epsilon\cdot((t-1)^{-1}+1)}, a_{\epsilon}\}$ are collinear.

Taking $\epsilon\cdot (t-1)^{-1}t$ instead of $\epsilon$ in \eqref{formula61} and \eqref{formula62}, we deduce that
$\{a_{\epsilon\cdot(t-1)^{-1}t }, b_{t\epsilon}, x_{t-1}\}$ and 
$\{a_{\epsilon\cdot(t-1)^{-1}t }, c_{t\epsilon}, x_{t}\}$ are collinear sets.  These two sets are respectively equal to 
$\{a_{\epsilon\cdot((t-1)^{-1}+1)} , b_{t\epsilon}, x_{t-1}\}$ and $\{a_{\epsilon\cdot((t-1)^{-1}+1)} , c_{t\epsilon}, x_{t}\}$.  
These,  \eqref{formula62}, and  \eqref{formula63} imply that
$$
\HP(a_0, x_{t-1},x_{t}; c_{t\epsilon}, b_{t\epsilon},
a_{\epsilon}, a_{\epsilon\cdot((t-1)^{-1}+1)})
$$
holds in $M$.  Therefore, there is a point $x_{t+1} \in M$ with 
$\Hc (a_0, x_{t}; x_{t-1}, x_{t+1} )$ such that 
$\{a_{\epsilon}, c_{t\epsilon}, x_{t+1}\}$ and  $\{a_0, b_0, c_0, x_2, \dots , x_{t+1}\}$  
are collinear.

Note that we have proved $\Hc(a_0,x_{t};x_{t-1},x_{t+1})$ for $ t=\psi(r) \text{ with } r \in \bbQ$. 

Proof of Part \eqref{tecnico61Part2}. From the proof of Part \eqref{tecnico61Part1} we know that $\hi\{a_0, b_{0}, c_0\} = \{a_0, b_{0}, c_0\} \cup \{ x_i \mid i\in \bbZ_{\geq0}\}$ is a countable conjugate sequence.
\end{proof}

Recall that $A$, $B$ and $C$ are the three lines of $L_0(\bV{\cdot}K_3)$ that are concurrent at $D$. 
Let $R_{\infty}=A,$ $R_{0}=B,$ $R_1=C$.  For $t = \psi(r) \in \bbZ_{>0}$, where $r \in\bbQ$,  we define $R_{t+1}$ recursively as 
$$\{ z : \exists\ a \in R_{\infty} \setminus D, \ x \in R_{t-1} \setminus D, \ y \in R_t \setminus D \text{ such that } \Hc(a,y;x,z) \} \cup \{D\}.$$
As in Section \ref{sec:PRinharmonic} we have the following properties.

\begin{cor}\label{cortecnico6} 
The set $R_{\psi(r)}$ is collinear for each $r \in\bbQ$.
\end{cor}

\begin{cor}\label{cortecnico56}  Let $g,h \in \bV$ be fixed.  Then for each $r \in\bbQ$, 
$R_{\psi(r)} \cap \hi(a_g,c_h) \ne \emptyset$.
\end{cor}

\begin{construction}[Projective Rectangle by Rational Harmonic Conjugation] \label{HarmonicConjPRQ}
Let $\PR:=(\cP,\mathcal{L},\mathcal{I})$, where 	
\begin{align*}
\cP \   &:={ \bigcup_{g,h \in \bV} \hi(a_g,c_h)} \cup \{D\} = R_{\infty} \cup {\bigcup_{i=0}^{\infty} R_i},\\
\cL  \ &:= \{R_i : i= \infty, 0, \dots, \} \cup \{\hi(a_g,c_h) :  g,h \in \bV\},
\end{align*}
and $\mathcal{I}$ is the incidence relation of set membership. 
\end{construction}

\begin{thm}\label{thm:Q}
The system $\PR$  is a projective rectangle of order $(|\bbQ|,|\bV|)$.
\end{thm}

\begin{proof}
The proof is essentially the same as that of Theorem \ref{tecnico3} for the field $\bbZ_p$.  The modifications due to having infinite sequences do not affect any of the harmonic relations or calculations.
\end{proof}

\sectionpage\section{A submatroid that generates a projective rectangle} \label{sec:minimalmatroid}

In this section we prove that a generalization of $R_{\mathrm{cycle}}[p]$ is
a matroid that extends by harmonic conjugation to the projective rectangle 
$\hi(L_p^k)$, where $p$ is a prime number and $k$ is a positive integer.  
We define the generalization in terms of any positive integer $m$.  
The matroid $L_m^k$, where $k$ and $m$ are positive integers, is defined exactly like $L_p^k$ in Section \ref{sec:harmonicmatroids}.    
Let $\{g_1, \dots, g_k \}$ be a minimal set of
generators of $\fG := (\bbZ_m)^k$. 
We set $G= \{0, g_1, \dots, g_k \}$. 
We write $R_{\mathrm{cycle}}^k[m]$ to mean any matroid isomorphic to the restriction
$$L_m^k | (A \cup B \cup \{c_f \mid f \in G \}).$$  
We call this the \emph{higher Reid cycle matroid}.  
\begin{figure}[htbp]
\begin{center}
\includegraphics[scale=.7]{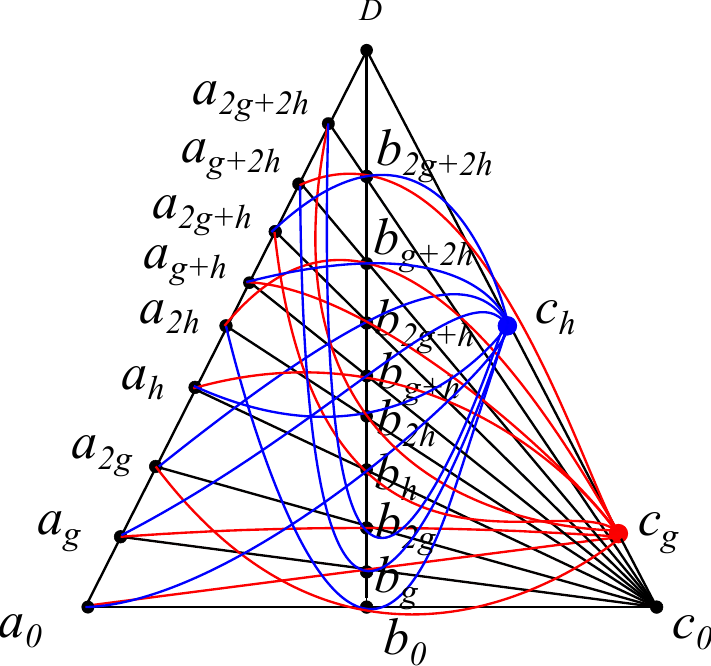}
\caption{The higher Reid cycle matroid $R_{\mathrm{cycle}}^2[3]$ with group $\bbZ_3\times\bbZ_3$; the identity element is $e=(0,0)$.}
\label{F:ReidCycle3x3}
\end{center}
\end{figure}

Gordon \cite{Gordon}, for $m=p$ (a prime) and Fl\'orez \cite{rflc} (for all $m$) found the algebraic representability of $R_{\mathrm{cycle}}[m]$.

\begin{lem}[{\cite[Theorem 2]{Gordon}} and {\cite[Theorem 2]{rflc}}] \label{Rembedding}
The Reid cycle matroid $R_{\mathrm{cycle}}[m]$ is algebraically representable, and is linearly representable, if and only if $m$ is a prime number.
\end{lem}

Fl\'orez \cite{rflc} also found that $L_p^k$ (recall that $p$ is prime) embeds in some harmonic matroid for every $k >0$.  More precisely, harmonic conjugation is involved when $k=1$:

\begin{lem} [{\cite[Lemma 9]{rflc}}] \label{mnembedding} Any embedding of $R_{\mathrm{cycle}}[p]$
in a harmonic matroid $M$ extends uniquely by harmonic conjugation to an embedding of $L_p$ in $M$.
\end{lem}

The proof of Lemma \ref{mnembedding} shows that there are points
$c_2, c_3, \dots , c_{p-1}$ that extend $R_{\mathrm{cycle}}[p]$ to $L_p$.
Those points satisfy harmonicity, $H(c_{j},D; c_{j-1}, c_{j+1})$,  and that
$\{a_i,b_{i+j},c_j\}$ is a collinear triple for $i, j \in \bbZ_p$.

We wish to extend these results to $R_{\mathrm{cycle}}^k[p]$.  Our objectives are the following two theorems as well as Proposition \ref{lpnembedding}.

\begin{thm}\label{hrprimenon}
For every $k>0$, the higher Reid cycle matroid $R_{\mathrm{cycle}}^k[m]$ is linearly representable, and is algebraically representable, if and only if $m$ is a prime number.  The same holds true for $L_m^k$.
\end{thm}

\begin{thm}\label{minimalmatroid}
Suppose that $L_p^k$ is embedded in a harmonic matroid $M$.  Then each $R_{\mathrm{cycle}}^k[p] \subseteq L_p^k$ is a submatroid of $L_p^k$ for which $\hi(R_{\mathrm{cycle}}^k[p])= \hi(L_p^k)$.  When $k=1$, $R_{\mathrm{cycle}}[p]$ is a minimal such matroid.
\end{thm}

\begin{qn}\label{Q:minimalmatroid}
What are the minimal submatroids $N$ of $R_{\mathrm{cycle}}^k[p]$ such that $\hi(N) = \hi(L_p^k)$?  It seems improbable that $R_{\mathrm{cycle}}^k[p]$ is minimal for $k>1$.
\end{qn}

We begin the proofs with the following lemma.

\begin{lem}\label{hrembeds}
The higher Reid cycle matroid $R_{\mathrm{cycle}}^k[p]$ is linearly and algebraically representable in characteristic $p$, and therefore embeds in some harmonic matroid, for every $k>0$.
\end{lem}

\begin{proof}
The matroid $R_{\mathrm{cycle}}^k[p]$ is linearly representable in the vector space $\bF^3$ over any field $\bF$ that contains $\mathbb{F}_{p^k}$, because $\fG \cong \mathbb{F}_{p^k}^+ \subseteq \bF^+$, the additive group of the field \cite[Theorem 4.1(a)]{b4}.  
Therefore, it embeds in two kinds of harmonic matroid: every projective space of dimension $2$ or greater over such a field, and consequently also in every full algebraic matroid of characteristic $p$ and transcendence degree at least equal to $3 = \text{rank} (R_{\mathrm{cycle}}^k[p])$ \cite[Proposition 6.7.10]{Oxley}.
\end{proof}

For use in the next proposition we identify the Reid cycle matroids contained in $L_p^k$.

\begin{construction}[Reid Cycle Submatroids]\label{subreid}
A Reid cycle submatroid $R_{\mathrm{cycle}}[p] \subseteq L_p^k$ must have the following form: a point set $A_1 \cup B_1 \cup C_1$ where $A_1 = \{a_{\alpha_i} \mid i \in \bbZ_p\} \cup \{D\}$, $B_1 = \{b_{\beta_j} \mid j \in \bbZ_p\} \cup \{D\}$, and $C_1 = \{c_\gamma, c_{\gamma'}, D\}$, where $\alpha_i, \beta_j, \gamma, \gamma' \in \fG$; and lines $\{a_{\alpha_i},b_{\beta_i},c_{\gamma}\}$ and $\{a_{\alpha_i},b_{\beta_{i+1}},c_{\gamma'}\}$ as well as $A_1$, $B_1$, and $C_1$.  (The point indexing can always be chosen this way because the index set $\bbZ_p$ is prime cyclic; we omit the details.)  The 3-point lines mean that $\beta_i = \alpha_i +\gamma$ and $\beta_{i+1} = \alpha_i +\gamma'$.  It follows that $\alpha_i +\gamma = \alpha_{i-1} +\gamma'$, so $\alpha_i = \alpha_{i-1} +\gamma' -\gamma = \alpha_0 +i(\gamma' -\gamma)$.  Let $\delta = \gamma' -\gamma$, so that $\gamma' = \delta +\gamma$.  Then $\alpha_i = \alpha + i\delta$ (where $\alpha$ is short for $\alpha_0$) and $\beta_i = \alpha + i\delta + \gamma$.  
We denote this submatroid by $R_p(\alpha,\delta,\gamma)$; see Figure \ref{F:Rp}.

\begin{figure}[htbp]
\begin{center}
\includegraphics[scale=.8]{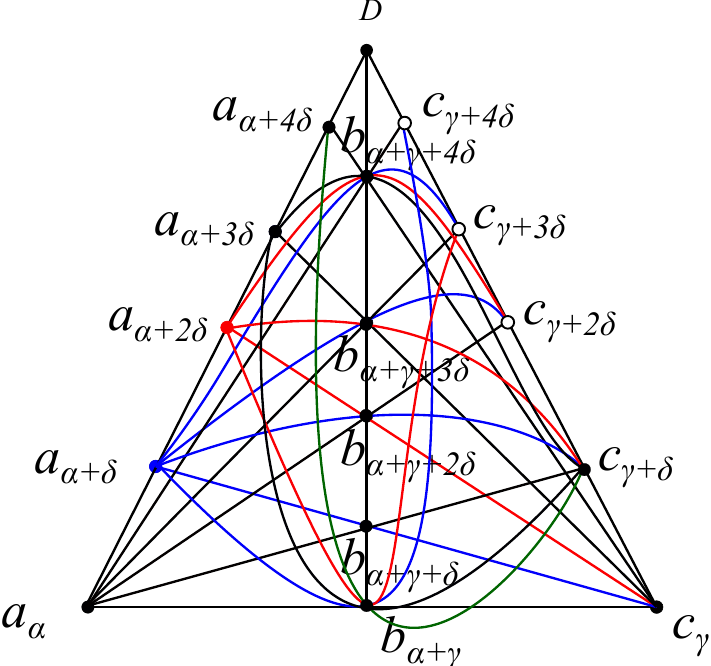}
\caption{The submatroid $R_p(\alpha,\delta,\gamma)$ (solid dots) and its harmonic extension $\bar R_p(\alpha,\delta,\gamma)$ (with the hollow dots).}
\label{F:Rp}
\end{center}
\end{figure}

Our description means there is a cyclic subgroup $\langle\delta\rangle = \{i\delta  \mid i \in \bbZ_p\} \cong \bbZ_p$, such that $\{\alpha_i\}_i = \alpha + \langle\delta\rangle$ and $\{\beta_i\}_i = \gamma +\alpha +\langle\delta\rangle$; that is, the group elements that label $A_1$ and $B_1$ are cosets of $\langle\delta\rangle$.  The choice of the parameters $\alpha, \gamma, \delta \in \fG$ is arbitrary as long as $\delta \neq 0$.  Every Reid cycle submatroid of $L_p^k$ has this form (up to changes of notation) if $p^k>2$ because the lines $A_1$ and $B_1$, having more than 3 points, must be contained in long lines of $L_p^k$.  (We note that the same Reid cycle submatroid may be constructed from different sets of parameters.)  
\end{construction}

\begin{prop}  \label{lpnembedding} Any embedding of
$R_{\mathrm{cycle}}^k[p]$ in a harmonic matroid $M$ extends uniquely by
harmonic conjugation  to an embedding of $L_p^k$ in $M$.
\end{prop}

\begin{proof}  By Lemma \ref{mnembedding} we may assume $k>1$.  Let $\alpha, \gamma, \delta$ be three fixed elements of $\fG$ with $\delta \ne 0$.  

We consider the simple matroid $R_p(\alpha,\delta,\gamma) \cong R_{\mathrm{cycle}}[p]$ of rank 3 on the point set 
$$E(R_p(\alpha,\delta,\gamma)) = \{a_{\alpha +t\delta}, b_{\gamma+\alpha +t\delta} \mid t\in \bbZ_p\} \cup \{ c_{\gamma}, c_{\gamma+\delta}, D \}.$$
Because $R_p(\alpha,\delta,\gamma)$ is a submatroid of $R_{\mathrm{cycle}}^k[p]$, it embeds in $M$. 
By Lemma \ref{mnembedding}, $R_p(\alpha,\delta,\gamma)$ extends by harmonic conjugation to an
embedding $\bar R_p(\alpha,\delta,\gamma)$ of $L_p$ in $M$, where the new points are denoted by
$c_{\gamma + t\delta}$ for $t \in \bbZ_p \setminus \{0,1\}$.
The existence of these new points does not depend on the choice of $\alpha$, but only on $\gamma$ and $\delta$, as we can see from their subscripts.
Moreover, the points satisfy $\HP (c_{\gamma \delta^{j}}, D;\ c_{\gamma + (j-1)\delta}, c_{\gamma + (j+1)\delta} )$ 
for every $ j \in \bbZ_p$ (as one can see in Figure \ref{F:Rp}), which shows directly that $\bar R_p(\alpha,\delta,\gamma)$ extends $R_p(\alpha,\delta,\gamma)$ by harmonic conjugation.

Identifying each point $a_i$, $b_i$ and $c_i$ in $L_p$ with
$a_{\alpha +i\delta}$, $b_{\gamma+\alpha +i\delta}$, $c_{\gamma +i\delta}$,
respectively, we can verify that the set of nontrivial flats of this
extension $\bar R_p(\alpha,\delta,\gamma)$ is
\begin{align*}\label{flats}
\cF(\alpha,\delta,\gamma)=\big\{ &\{ a_{\alpha +t\delta} \mid t\in \bbZ_p \} \cup \{ D \},
\{ c_{\gamma +t\delta} \mid t\in \bbZ_p \} \cup \{ D
\}, \\&\quad
\{ b_{\gamma +\alpha +t\delta} \mid t\in \bbZ_p \} \cup \{ D \},
\{a_{\alpha + i\delta}, b_{\gamma +\alpha +(i+j)\delta}, c_{\gamma +j\delta} \}
\text{ for } i, j \in \bbZ_p \big\}.
\end{align*}

Varying $\alpha$ in $\fG$ we obtain a family of matroids, $\{ R_p(\alpha,\delta,\gamma) \}_\alpha$.  
The matroid $M_p^k(\gamma, \delta)$ defined on ground set 
$$\bigcup_{\alpha \in \fG} E(\bar R_p(\alpha,\delta,\gamma)) = E (R_{\mathrm{cycle}}^k[p] ) \cup \{ c_{\gamma +t\delta} \mid t \in \bbZ_p \}$$ 
with the set of nontrivial flats
\begin{align*}
\cF^k(\gamma, \delta) 
&=\bigcup_{\alpha \in \fG} \cF(\alpha,\delta,\gamma)
= \{F \mid F \text{ is a flat of } R_p(\alpha,\delta,\gamma) \text{ for some } \alpha \in \fG \} 
\end{align*}
is a (harmonic) extension of the matroid $R_{\mathrm{cycle}}^k[p]$ and a submatroid of $L_p^k$.

The matroid $M_p^k$ with the ground set
$$
\bigcup_{\gamma, \delta \in G} E\big(M_p^k(\gamma,\delta)\big) =
E \big(R_{\mathrm{cycle}}^k[p] \big) \cup 
\{ c_{\omega} \mid \omega \in \fG\setminus G \}=A\cup B \cup C
$$
in the harmonic matroid $M$, defined by the set of nontrivial flats
\begin{align*}
\bigcup_{\gamma, \delta \in G} \cF^k(\gamma, \delta) &= \{F \mid F \text{ is
a flat of } M_p^k(\gamma,\delta) \text{ for some }\gamma, \delta \in G \}
\\
&=\big\{ \{ a_{\alpha} \mid \alpha \in \fG
\} \cup \{ D \}, \{ c_{\alpha} \mid \alpha \in \fG \} \cup \{D \},  \{
b_{\alpha} \mid \alpha \in \fG \} \cup \{ D \}, \\
&\qquad \{a_{\alpha}, b_{\alpha +\gamma+\delta}, c_{\gamma+\delta} \} \text{ for } \alpha, \gamma, \delta \text{ in } \fG \big\},
\end{align*}
is a (harmonic) extension of $R_{\mathrm{cycle}}[p]$ in $M$.  Its points and flats are the points and flats of $L_p^k$.  Thus, $M_p^k = L_p^k$, which proves the proposition.
\end{proof}

If $L_p^k$ is embedded in a harmonic matroid, then any submatroid $R_{\mathrm{cycle}}^k[p]$ naturally is also embedded.  The following corollary follows from Proposition \ref{lpnembedding} and the fact that $\hi$ is a closure operator.

\begin{cor}\label{MinimalMatroid} Suppose that $R_{\mathrm{cycle}}^k[p]$ is embedded in a harmonic matroid
$M$.  Then $\hi(R_{\mathrm{cycle}}^k[p])= \hi( L_p^k)$.
\end{cor}

 \begin{proof}[Proof of Theorem \ref{hrprimenon}]
Because $R_{\mathrm{cycle}}^k[m]$ contains a copy of $R_{\mathrm{cycle}}[m]$, it can only be linearly or algebraically representable when $R_{\mathrm{cycle}}[m]$ is.  By Lemma \ref{Rembedding}, the latter is not embeddable when $m$ is a composite number.  That rules out embedding of $R_{\mathrm{cycle}}^k[m]$ unless $m$ is prime.  When $m$ is prime we call upon Lemma \ref{hrembeds}.
 \end{proof}

\begin{proof}[Proof of Theorem \ref{minimalmatroid}] 
From Proposition \ref{lpnembedding}, Corollary \ref{MinimalMatroid}, and Theorem \ref{ThmHarmonicPR} we conclude that $\hi(R_{\mathrm{cycle}}^k[p])= \hi( L_p^k)$. 

Regarding $R_{\mathrm{cycle}}[p]$, Gordon \cite[Proof of Proposition 12]{Gordon} proved that any proper minor embeds in a full algebraic matroid over the rational numbers $\bbQ$.  (He derived this from the similar result of Lindstr\"om for linear representability.)  Since a full algebraic matroid is harmonic, $\hi(R)$ is algebraically representable over $\bbQ$.  If $\hi(R) = \hi(R_{\mathrm{cycle}}[p])$, then it equals $\hi(L_p)$, which contains $L_p$, so $L_p$ would be algebraically representable over $\bbQ$.  However, $L_p$ is algebraically representable only over fields of characteristic $p$ \cite[Theorem 7]{rflc}.  This contradiction proves that $\hi(R)$ is smaller than $\hi(L_p)$.
\end{proof}

Define the \emph {algebraic characteristic set} $\chi_{A} (M)$ to be the set of field characteristics over which the matroid $M$ is algebraic. 

\begin{prop} \label{CharSet}  The algebraic characteristic sets $\chi_{A}(\hi(L_p^k))=R_{\mathrm{cycle}}^k[p]=\{p\}$.
\end{prop}

\begin{proof} 
From \cite{b4} we know that $L_p^k$ is linearly representable over a field of characteristic $p$. Therefore, by \cite[Proposition 6.7.10]{Oxley}, $L_p^k$ is algebraically representable over a field of characteristic $p$. Consequently, it embeds in a full algebraic matroid $M$ over a field of characteristic $p$, which is a harmonic matroid by Lindstr\"om \cite{lhc}.  We deduce that $\hi(L_p^k)$ is algebraically representable over a field of characteristic $p$.  This proves that $p \in \chi_{A}(\hi(L_p^k))$.

On the other hand, from \cite[Theorem 2]{Gordon} we know that $\chi_{A}(R_{\mathrm{cycle}}[p])=\{p\}$. Since $R_{\mathrm{cycle}}[p] \subseteq R_{\mathrm{cycle}}^k[p] \subseteq \hi(L_p^k)$, we obtain the desired conclusion.
\end{proof}

\sectionpage

\end{document}